\documentclass{article}
\usepackage{graphicx,amsmath,amsbsy,amssymb,textcomp,amsthm,amsfonts,xcolor} 
\usepackage{todonotes}
\usepackage{geometry} 
\usepackage{comment}
\usepackage{pgfplots}
\usepackage{stmaryrd}
\usepackage{hyperref}
\hypersetup{
    colorlinks=true,
    linkcolor=blue,
    filecolor=magenta,      
    urlcolor=cyan
    }
\usepackage[justification=centering]{caption}
\SetSymbolFont{stmry}{bold}{U}{stmry}{m}{n}
\usepackage{dsfont}
\newcommand{\hz}{\hat{\zeta}}
\newcommand{\mui}{\mu^{(i)}}
\newcommand{\phii}{\phi^{(i)}}
\newcommand{\be}{\mathbf{e}}
\newcommand{\N}{\mathbb{N}}
\newcommand{\R}{\mathbb{R}}
\newcommand{\Z}{\mathbb{Z}}
\newcommand{\E}{\mathbb{E}}

\newcommand{\bb}{\textbf{b}}
\newcommand{\bw}{\textbf{w}}
\newcommand{\bx}{\textbf{x}}
\newcommand{\cM}{\mathcal{M}}

\newcommand{\cU}{\mathcal{U}}
\newcommand{\cT}{\mathcal{T}}

\newcommand{\cW}{\mathcal{W}}
\newcommand{\bT}{\mathbb{T}}

\newcommand{\bze}{\boldsymbol{\zeta}}

\newcommand{\tbze}{\boldsymbol{\tilde{\zeta}}}

\newcommand{\bmu}{\boldsymbol{\mu}}
\newcommand{\btheta}{{\boldsymbol{\theta}}}
\newcommand{\bN}{\mathbf{N}}
\newcommand{\bzero}{\mathbf{0}}
\newcommand{\bone}{\mathbf{1}}
\newcommand{\bg}{\mathbf{g}}
\newcommand{\bn}{\mathbf{n}}
\newcommand{\bk}{\mathbf{k}}
\newcommand{\Prob}{\mathbb{P}}
\newtheorem{Theorem}{Theorem}[section]
\newtheorem{Proposition}[Theorem]{Proposition}
\newtheorem{Lemma}[Theorem]{Lemma}
\newtheorem{Definition}[Theorem]{Definition}
\newtheorem{Remark}[Theorem]{Remark}
\newtheorem{Corollary}[Theorem]{Corollary}

\usetikzlibrary{patterns}

\pgfplotsset{compat=1.18}

\title{Existence of critical tiltings and local limits of general size-conditioned Bienaymé-Galton-Watson multitype trees}
\author{Rémy Poudevigne, Paul Thévenin}
\date{\today}

\begin{document}
\maketitle
\abstract{
We are interested in the structure of multitype Bienaymé-Galton-Watson (BGW) trees conditioned on integer linear combinations of the numbers of vertices of given types. We show that, under regularity assumptions on the offspring distributions, it is always possible to find a critical BGW tree having the same conditional distribution. This allows us to prove the existence of local limits for noncritical BGW trees, under a large variety of conditionings. Our proof is based on geometric considerations on the set of the so-called exponential tiltings of a family of offspring distributions.
}

\tableofcontents

\section{Introduction}

Multitype Bienaymé-Galton-Watson trees (or, in short, BGW trees) are a model of random trees, where each vertex is given a type and has offspring according to a distribution that only depends on its type. The asymptotic structure of large BGW trees under various conditionings (total size, number of leaves, number of vertices of a given type) has attracted a lot of interest (see e.g. \cite{Kor12, Riz15, Add19, The20, KM23, ABD23}), in order to tackle questions coming from theoretical mathematics, biology or physics. 

We consider here trees whose set of types is $[K] := \{1, \ldots,K\}$.
In the monotype case ($K=1$), which has been extensively studied, let $\mu$ be a distribution on $\Z_+$, assumed to be critical (that is, of mean $1$). Let $\cT_n$ be the BGW tree with offspring distribution $\mu$, conditioned on having $n$ vertices, provided that this occurs with positive probability. Abraham and Delmas \cite{AD14} proved that the tree $\cT_n$ converges in a local sense to a random discrete infinite tree $\cT_*$; that is, for all $k \geq 1$, the $k$-neighbourhood of the root of $\cT_n$ converges in distribution to the $k$-neighbourhood of the root of $\cT_*$. The tree $\cT_*$ is called Kesten's tree, and was first characterized in \cite{Kes86}.
In another direction, under the additional assumption that $\mu$ has finite variance $\sigma^2$, Aldous \cite{Ald91a,Ald91b,Ald93} proved that $\cT_n$, seen as a metric space where all edges have length $(\sigma\sqrt{n})^{-1}$, converges in distribution to a random compact limit metric space called the Brownian Continuum Random Tree (in short, CRT) $\cT_\infty$. This tree is referred to as the scaling limit (i.e. limit after rescaling) of $\cT_n$.

In the case where $K \geq 2$, similar partial results were obtained. Miermont \cite{Mie08} proved that the BGW tree $\cT^{(i)}_n$, which is the multitype tree with root of type $i$ conditioned on having $n$ vertices of type $i$, admits the CRT as its scaling limit, given that the offspring distribution is critical (in a multitype sense that we define later) and has small exponential moments. The latter assumption was then lifted by Haas and Stephenson \cite{HS21}, who only assume a finite variance condition. 

On the other hand, let $(\gamma_1, \ldots, \gamma_K) \in \Z_+^K \backslash \{ 0, \ldots, 0\}$. Stephenson \cite{Ste18} showed, under an assumption of small exponential moments, the local convergence of a critical multitype BGW tree $\cT$ conditioned on $\sum_{i=1}^K \gamma_i N_i(\cT)=n$ (for $n$ going to $+\infty$ along a subsequence on which these events have positive probability) towards a multitype version of Kesten's tree (see \cite[Definition 2.7]{Ste18}). Here and everywhere in the paper, $N_i(\cT)$ stands for the number of vertices of type $i$ in the tree $\cT$.

In all the previously mentioned results, the criticality assumption turns out to be fundamental, and proving results outside of criticality can be challenging. In the monotype case, Janson \cite{Jan12} proved the following: a (non necessarily critical) distribution $\mu$ on $\Z_+$ being given, under some regularity assumptions, there exists a unique critical distribution $\tilde{\mu}$ such that, for all $n$, the $\mu$-BGW tree $\cT_n$ has the same distribution as the $\tilde{\mu}$-BGW tree $\tilde{\cT}_n$ (again, conditioned to have the same number $n$ of vertices). This allows to prove the local and scaling convergences of the noncritical tree $\cT_n$ as well.

Little is known in the multitype case, outside of criticality. Let us mention Pénisson \cite{Pen16}, who showed that - under some regularity assumptions - for all $n_1, \ldots, n_K$, the tree $\cT_{n_1, \ldots, n_K}$ (the BGW tree conditioned to have $N_i(\cT)=n_i$ vertices of type $i$ for all $i \in [K]$) is distributed as a critical BGW tree under the same conditioning. Recently, Abraham, Delmas and Guo \cite{ADG18} showed the local convergence of the tree $\cT_{n_1, \ldots, n_K}$, provided that the asymptotic proportions of vertices of each type satisfy a specific relation. In another direction, Abraham, Bi and Delmas \cite{ABD23} considered the case where the multitype tree has the structure of a monotype BGW tree, and where the type of a vertex only depends on its number of children. In this specific case, they characterized the local limit of the tree $\cT_{n_1, \ldots, n_K}$ as $n \rightarrow \infty$, where $n_i \sim \alpha_i n$ for all $i$, for some $(\alpha_1, \ldots, \alpha_K) \in [0,1]^K$. Outside of criticality, they managed to obtain a detailed picture similar to Janson's in the monotype case, and to characterize the cases in which the limit is a nondegenerate multitype Kesten tree.

Finally, in the general multitype case and under a general size-conditioning, the second author obtains partial results \cite{The23}: the size-conditioned tree $\cT$ is distributed as a critical tree $\tilde{\cT}$ under the same size-conditioning, under some rather strong assumptions on the generating functions of the offspring distributions.

Our aim in this paper is to extend all these results, and generalize \cite{The23}.

\paragraph{Notation}

In all the paper, we denote by $\Z$ the set of integers, by $\Z_+:= \{0,1, \ldots \}$ the set of nonnegative integers, and by $\N := \Z_+ \backslash \{0 \}$ the set of positive integers. For all $K \in \N$, we denote by $[K]$ the set $\{1,\ldots,K\}$. 

The space $\R^d$ is endowed with the usual Euclidean metric associated with the norm $||\cdot||_2$. For $r>0$, $\mathbf{x} \in \R^K$, we let $B_{r}(\mathbf{x})$ be the open ball of radius $r$ in $\R^K$ centered at $\mathbf{x}$. For any $\btheta := (\theta_1, \ldots, \theta_d) \in \R^d$, we set $e^{\btheta}:=(e^{\theta_1},\ldots,e^{\theta_d})$.

For any $a,b \geq 1$ and $A=\Z$ or $\Z_+$, $\cM^*_{a,b}(A)$ denotes the set of $a \times b$ matrices with coefficients in $A$, without the null matrix. We also denote by $\bzero \in \R^K$ (and $\bone \in \R^K$ respectively) the $K$-dimensional vectors with all coordinates equal to $0$ (and $1$ respectively), which will be line- or column-vectors depending on the context. The transpose of a matrix $M$ is denoted by $M^{\intercal}$, its image by $\text{Im}(M)$ and its trace by $Tr(M)$.
For any $d \geq 1$, $\mathbf{a}:=(a_1, \ldots, a_d)^\intercal$ and $\mathbf{b} := (b_1, \ldots, b_d)^\intercal \in \R^d$, we denote by $\mathbf{a} \cdot \mathbf{b} := \sum_{i=1}^d a_i b_i$ the usual scalar product on $\R^d$.

For $\mu$ a distribution, we denote by $\Prob_{\mu}$ the probability measure under $\mu$ and $\E_\mu$ the expectation with respect to $\mu$.

\paragraph{Outline of the paper}

We start by recalling definitions and results on trees in Section \ref{sec:trees}. Section \ref{sec:results} is devoted to the statement of our main results concerning the existence of the so-called critical tiltings equivalent to a given measure. We study in Section \ref{sec:properties of Mcrit} the properties of the set of critical tiltings, while Section \ref{sec:propertiesofthefunctionc} concerns the properties of a function $\chi$, which plays an important role in the proof of our main results. This proof is done in Section \ref{sec:proof}. Section \ref{sec:locscallim} shows how our results allow us to prove the existence of a local limit for noncritical trees under a specific conditioning, while Section \ref{sec:appendix} provides a counterexample of our results in the case of non entire generating functions. Finally, Section \ref{sec:appenonze} shows that the function $\chi$ is not bijective, showing in addition that there are elements with arbitrarily many pre-images.

\section{Background on random trees}
\label{sec:trees}

We recall here the definitions of multitype plane trees and multitype BGW trees. 

\subsection{Plane trees}

We first recall the definition of plane trees using Neveu's formalism \cite{Nev86}. Let $\cU := \bigcup_{k \geq 0} \N^k$ represent the set of finite sequences of positive integers, where we adopt the convention that $\N^0=\{\varnothing\}$. To simplify notation, for $k \in \N$, an element $u$ of $\N^k$ is denoted as $u=u_1 \cdots u_k$, with $u_1, \ldots, u_k \in \N$. For $k \in \Z_+$, $u=u_1\cdots u_k \in \N^k$, and $i \in \N$, we use $ui$ to represent the element $u_1 \cdots u_ki \in \N^{k+1}$ and $iu$ to represent the element $iu_1 \cdots u_k \in \N^{k+1}$, with the convention that $\emptyset i = i \emptyset=i$ for $i \in \N$.

A plane tree $t$ is defined as a subset of $\cU$ satisfying three conditions:
\begin{itemize}
\item(i) $\varnothing \in t$ (indicating the presence of a root in the tree);
\item(ii) if $u=u_1\cdots u_n \in t$, then, for all $k \leq n$, $u_1\cdots u_k \in t$ (these elements are referred to as ancestors of $u$);
\item(iii) for any $u \in t$, there exists a nonnegative integer $k_u(t)$ such that, for every $i \in \N$, $ui \in t$ if and only if $1 \leq i \leq k_u(t)$ ($k_u(t)$ is called the number of children of $u$, or the outdegree of $u$, $u$ is called the parent of $ui$ and $ui$ is called a child of $u$).
\end{itemize}

The elements of $t$ are referred to as vertices, and we denote the total number of vertices of $t$ by $|t|$. Finally, we denote the set of plane trees by $\bT$. We always consider a plane tree $t$ as a metric space by connecting each non-root vertex to its parent, with the usual graph distance $d_t$ on it (that is, all edges have length $1$. For $a>0$, we also define $a d_t$ as the distance on $t$ for which all edges have length $a$. Furthermore, we endow a plane tree $t$ with the uniform measure $m_t$ on its vertices.

\subsection{Multitype plane trees}

Fix $K \in \N$. A $K$-type plane tree is a pair $T := (t,\be_t)$ where $t \in \bT$ is a plane tree and $\be_t: t \mapsto [K]$. For $u \in t$, $\be_t(u)$ is called the type of the vertex $u$. We let $\bT^{(K)}$ be the set of $K$-type plane trees, and, for $i \in [K]$, we denote by $\bT^{(K,i)}$ the subset of $\bT^{(K)}$ of trees whose root has label $\be_t(\varnothing)=i$. We also denote by $N_i(T)$ the number of vertices $u$ of the tree $t$ which are of type $i$, that is, such that $\be_t(u)=i$, and we set $\mathbf{N}(T) := (N_1(T), \ldots, N_K(T))^{\intercal}$. Finally, we call $t$ the \textit{shape} of the multitype tree $T$.

\subsection{Multitype BGW trees}

Let us now define $K$-type BGW trees. For $K \in \N$, consider the set $\cW_K := \bigcup_{n \geq 0} [K]^n$. Let $\bze := (\zeta^{(i)})_{i \in [K]}$ be a family of probability distributions on $\cW_K$. For $i \in [K]$, we define a probability distribution on $\bT^{(K,i)}$ as follows. Consider a family of independent variables $(X_u^i, u \in \cU, i \in [K])$ with values in $\cW_K$, such that for all $(u,i) \in \cU \times [K]$, $X_u^i$ follows the distribution $\zeta^{(i)}$. We recursively construct a random $K$-type tree $\cT^{(i)} := (t, \be_t) \in \bT^{(K,i)}$ as follows. First, $\varnothing \in t$ and $\be_t(\varnothing) = i$; then, if $u \in t$ and $\be_t(u) = j$, then, for $k \in \N$, $uk \in t$ if and only if $1\leq k \leq |X_u^j|$ and in this case $\be_t(uk)=X_u^j(k)$.

In other words, the root of $\cT^{(i)}$ has type $i$, and vertices of type $j$ in $\cT^{(i)}$ have children independently according to $\zeta^{(j)}$. We refer to $\cT^{(i)}$ as a $\bze$-BGW tree with root type $i$.

An objet of importance is the \textit{projection} of the family $\bze$. For any $w \in \cW_K$ and $j \in [K]$, let $w^{(j)}$ be the number of $j$'s in $w$. We call projection of $w$ the element $p(w) = (w^{(1)}, \ldots, w^{(K)}) \in \Z_+^K$. For $i \in [K]$, denote by $\mu^{(i)}$ the probability distribution on $\Z_+^K$ defined as follows: for all $(k_1, \ldots, k_K) \in \Z_+^K$,
\begin{align*}
\mu^{(i)}(k_1,\ldots,k_K)=\sum_{\substack{w \in \cW^K \\ p(w)=(k_1,\ldots,k_K)}} \zeta^{(i)}(w).
\end{align*}

The family $\bmu:=(\mu^{(i)}, i \in [K])$ is called the \textit{projection} of $\bze$.

The generating function $\phi^{(i)}$ of $\mu^{(i)}$ is the series
\begin{align*}
    \phi^{(i)}(x_1, \ldots, x_K) = \sum_{(k_1, \ldots, k_K) \in \Z_+^K} \mu^{(i)}(k_1, \ldots, k_K) x_1^{k_1} \cdots x_K^{k_K}.
\end{align*}

We also define the mean matrix $M = (m_{i,j})_{i,j \in [K]}$ as the $K \times K$ matrix such that
\begin{align*}
m_{i,j} = \sum_{\textbf{z} \in \Z_+^K} z_j \mu^{(i)}(\textbf{z}),
\end{align*}
with the convenient notation $\textbf{z} := (z_1, \ldots, z_K) \in \Z_+^K$. In other words, $m_{i,j}$ is the expected number of children of type $j$ of a vertex of type $i$.

Numerous asymptotic structural properties of $\cT^{(i)}$ depend only on the projection $\bmu$. Let us mention some of them.

\begin{itemize}  
    \item(Critical) We say that $\bmu$ is critical if the spectral radius $\rho(M)$ of $M$ is equal to $1$, subcritical if $\rho(M) \leq 1$, and supercritical if $\rho(M) \geq 1$.
    \item(Entire) We say that $\bmu$ is entire if, for all $i$, the generating function $\phii$ of $\mui$ is entire.
    \item(Finite) A family $\bmu$ is called finite if, for all $i\in[K]$:
        \begin{equation}
        \label{eq:finias}
        \mathbb{P}\left( |\cT^{(i)}| < \infty\right) > 0,
        \end{equation}
    where $\cT^{(i)}$ is a $\bze$-BGW tree for some $\bze$ whose projection is $\bmu$.
    \item(Nondegenerate) We say that $\bmu$ is nondegenerate if it satisfies:
\begin{equation}
\label{eq:nondegenerate}
\exists i \in [K], \mui\left( \left\{ \textbf{z}, \sum_{j \in [K]} z_j \geq 2 \right\} \right) > 0.
\end{equation}
    \item(Nonlocalized) A projection $\bmu$ is called nonlocalized if, for all $X\in\R^K \backslash \{\bzero\}$, there exists $i\in[K]$ such that, if $(k_1, \ldots, k_K)$ is distributed according to $\mu^{(i)}$, then $\sum_{j=1}^K X_j k_j$ is not deterministic.   
    \item(Irreducible) We say that $\bmu$ is irreducible if $M$ is irreducible, that is, for all $i,j \in [K]$ there exists $p \in \N$ such that $M^p_{i,j} > 0$.
\end{itemize}

We also say that a family $\bze$ is critical (resp. supercritical, subcritical, entire, finite, nondegenerate, nonlocalized, irreducible) if its projection $\bmu$ is.

The Perron-Frobenius theorem states that, when $\bmu$ is irreducible, $\rho(M)>0$ is an eigenvalue of $M$ and is simple. Up to multiplicative constants, its left and right $\rho(M)$-eigenvectors (respectively $\mathbf{a}$ and $\mathbf{b}$) are unique and have positive coefficients. We always normalize them so that $$\sum_{i \in [K]} a_i = \sum_{i \in [K]} a_i b_i = 1.$$

\subsection{Size-conditioned BGW trees}

Our main object of interest is multitype BGW trees conditioned on their size being large, for a rather general notion of size. 

\begin{Definition}[Size-conditioned tree]
    Let $\ell \geq 1$ and $\Gamma \in \cM^*_{\ell,K}(\Z)$. Let also $\bg := (g_1, \ldots, g_\ell) \in \Z^\ell$. Let $\bze$ be a family of distributions on $\cW_K$. We define the tree $\cT^{(i)}_{\Gamma, \bg}$ as the tree $\cT^{(i)}$ conditioned on 
    \begin{equation}
    \label{eq:condition}
        \Gamma \left( \begin{array}{c}
           N_1\left( \cT^{(i)} \right)   \\
             \vdots \\
             N_K\left( \cT^{(i)} \right) 
        \end{array} \right) = 
        \left(\begin{array}{c}
             g_1 \\
             \vdots \\
             g_\ell
        \end{array}\right),
    \end{equation}
    provided that \eqref{eq:condition} holds with positive probability.
\end{Definition}

One of the main goals of the paper is to answer the following question: does there exist a multitype BGW tree $\tilde{\cT}^{(i)}$ whose offspring distribution is critical, and such that, for all $\bg$ such that \eqref{eq:condition} occurs with positive probability, we have the equality in distribution:
\begin{equation*}
    \cT^{(i)}_{\Gamma, \bg} \overset{(d)}{=} \tilde{\cT}^{(i)}_{\Gamma, \bg}?
\end{equation*}

When it is the case, then, the same way as Janson \cite{Jan12} in the monotype case, we can understand the behaviour of the tree $\cT^{(i)}_{\Gamma, \bg}$ as $|g_1|, \ldots, |g_K| \rightarrow \infty$ by studying the critical tree $\tilde{\cT}^{(i)}$ instead of $\cT^{(i)}$.

\subsection{Local limit of trees}

We say that a tree $T$ is locally finite if, for all $x \in T$, for all $r>0$, the number of vertices of $T$ at distance $\leq r$ from $x$ is finite.

\begin{Definition}
Let $(\cT_n)_{n \geq 1}$ be a sequence of (random) finite trees of respective roots $(\varnothing_n)_{n \geq 1}$, and $\cT_*$ be an a.s. locally finite infinite tree of root $\varnothing_*$. We say that $\cT_*$ is the local limit of the sequence $(\cT_n)_{n \geq 1}$ if, for any fixed $r \geq 1$, the following convergence holds in distribution:
\begin{equation*}
    B_r(\cT_n) \underset{n \rightarrow \infty}{\overset{(d)}{\rightarrow}} B_r(\cT_*),
\end{equation*}
where $B_r(\cT_n)$ is the ball of radius $r$ around $\varnothing_n$ in $\cT_n$, and $B_r(\cT_*)$ is the ball of radius $r$ around $\varnothing_*$ in $\cT_*$. We write, in this case:
\begin{align*}
    \cT_n \underset{n \rightarrow\infty}{\overset{(d),loc}{\rightarrow}} \cT_*.
\end{align*}
\end{Definition}

\section{Statement of the results}
\label{sec:results}

We define here our main object of study, which is a family of projections that we call exponential tiltings of a projection $\bmu$. Our main tool is a function $\chi$ which we define in Section \ref{ssec:function chi}, before stating our main results in Section \ref{ssec:main results}.

\subsection{Exponential tiltings and asymptotic directions}

In this paper, being given a projection $\bmu$, we are interested in a certain family of projections, which we call exponential tiltings of $\bmu$.

\begin{Definition}[Exponential tiltings]
Let $\bmu := (\mu^{(i)})_{i\in [K]}$ be a family of probability measures on $\Z_+^K$. For any fixed $\btheta:=(\theta_1, \ldots, \theta_K) \in \R^K$ we define the measures $(\mu_{\btheta}^{(i)})_{i\in [K]}$ as follows: for all $i \in [K]$, for all $k_1, \ldots, k_K \in \Z_+$:
\begin{equation}
\label{eq:equivalent tiltings}
\mu_{\btheta}^{(i)}(k_1,\dots,k_K) = \mu^{(i)}(k_1,\dots,k_K)\frac{e^{\sum_{i=1}^K \theta_i k_i}}{\E_{\mu^{(i)}}(e^{\btheta \cdot \mathbf{k}})},
\end{equation}
where in the expectation $\mathbf{k}:=(k_1, \ldots, k_K)$ is distributed according to $\mu^{(i)}$. We set $\bmu_{\btheta}:=(\mu^{(1)}_{\btheta}, \ldots, \mu^{(K)}_{\btheta})$, and call it an exponential tilting of $\bmu$. 
\end{Definition}

Observe that the tilting operation is consistent with the projection. Indeed, if $\bmu$ is the projection of a family $\bze$, for any $\btheta \in \R^K$, we can define the exponential tilting $\bze_{\btheta}$ of $\bze$ the same way as $\bmu_{\btheta}$ by setting, for all $\bw \in \cW_K$, all $i \in [K]$:

\begin{align*}
    \zeta^{(i)}_{\btheta}(\bw) = \zeta^{(i)}(\bw) \frac{e^{\sum_{i=1}^K \theta_i w^{(i)}}}{\sum_{\tilde{\bw} \in \cW_K} \zeta^{(i)}(\tilde{\bw}) e^{\sum_{i=1}^K \theta_i \tilde{w}^{(i)}}}.
\end{align*}

We let $\phi_\btheta^{(i)}$ be the generating function of $\mu_\btheta^{(i)}$, that is, $\phi_\btheta^{(i)}: \R^K \mapsto \R$ is such that
\begin{equation*}
\phi_\btheta^{(i)}(x_1, \ldots, x_K) = \sum_{k_1, \ldots, k_K \in \Z_+} \mu^{(i)}_\btheta( k_1, \ldots, k_K) x_1^{k_1} \cdots x_K^{k_K},
\end{equation*}
and we set $\phi^{(i)} := \phi_\bzero^{(i)}$ for all $i \in [K]$.

Note that, then, \eqref{eq:equivalent tiltings} rewrites: for all $\btheta \in \R^K$, all $x_1, \ldots, x_K \in \R^K$,
\begin{align*}
\phi_\btheta^{(i)}(x_1, \ldots, x_K) = \frac{\phi^{(i)}\left( x_1 e^{\theta_1}, \ldots, x_K e^{\theta_K} \right)}{\phi^{(i)}\left( e^{\theta_1}, \ldots, e^{\theta_K} \right)},
\end{align*}
if these quantities are well-defined.

We denote by $M_\btheta$ the mean matrix of $\left( \mu^{(i)}_\btheta \right)_{i \in [K]}$ and by $\rho_\btheta$ the spectral radius of $M_\btheta$. For all $i \in [K]$, we also define $\cT^{\btheta,(i)}$ as the multitype BGW tree with distribution $\bmu_{\btheta}$ and root of type $i$.\\
When $\btheta$ is critical, it turns out that there is a favoured direction that we call the asymptotic direction.
\begin{Definition}[Asymptotic direction]
\label{def:asymdir}
    For any critical and irreducible measure $\bmu_\btheta$, we call \textit{asymptotic direction} the (renormalized) $1$-left eigenvector $X_\btheta$ of $M_\btheta$, that is, the vector $X_\btheta \in \R^K$ with positive coordinates such that $X_{\btheta}^{\intercal}(M_{\btheta}-I_K)=\bzero$, renormalized so that $||X_\btheta||_1=1$. This vector exists and is unique by the Perron-Frobenius theorem.
\end{Definition}

\subsection{The function $\chi$ and $\Gamma$-equivalence}
\label{ssec:function chi}

We are particularly interested in an equivalence relation on the set of  exponential tiltings. To define this equivalence relation, we first introduce a function $\chi$, similar to the log-Laplace transform of the functions $(\phi^{(i)})_{1\leq i\leq K}$ which will play an important role throughout the paper.

\begin{Definition}
\label{def:functionc}
    The function $\chi: \R^K \rightarrow \R^K$ is defined by $\chi(\btheta)=(\chi_1(\btheta), \ldots, \chi_K(\btheta))^{\intercal} \in \R^K$ where, for all $i \in [K]$:
\[
\chi_i:\left\{ 
\begin{array}{ll}
 \mathbb{R}^K \xrightarrow{} \mathbb{R} \\
 \btheta  \mapsto \log (\phi^{(i)}(e^\btheta))-\theta_i,
\end{array}
\right.
\]
where we recall that $e^\btheta := (e^{\theta_1}, \ldots, e^{\theta_K})$.\\
\noindent We call $\mathcal{C}_{image}:=\chi(\R^K)$ the image of $\R^K$ by the function $\chi$.
\end{Definition} 

Observe that the function $\chi$ is not a bijection (see Appendix \ref{sec:appenonze}). Its interest is twofold: first, it naturally appears in the following definition of equivalent tiltings; second, its Jacobian is equal to $M_{\theta}-I_K$, which helps provide a characterization of critical measures in terms of $\chi$.

We now introduce the notion of equivalent tiltings, which are tiltings under which the distribution of the size-conditioned tree does not vary.

\begin{Definition}[$\Gamma$-equivalent tiltings]
\label{def:eqtilt}
    Let $\Gamma \in \cM_{\ell,K}(\Z)$, for some $\ell \geq 1$. We say that $\btheta, \btheta' \in \R^K$ are $\Gamma$-equivalent if
    \begin{align*}
        \chi(\btheta)-\chi(\btheta') \in \left( \ker \, \Gamma \right)^\perp.
    \end{align*}
    If this holds, we write $\btheta \sim_\Gamma \btheta'$.
\end{Definition}
For convenience and with a slight abuse of notation, we will sometimes write that $\bmu_{\btheta}$ is $\Gamma$-equivalent to $\bmu_{\btheta'}$ whenever $\btheta \sim_\Gamma \btheta'$.\\

In the rest of the paper, unless explicitly mentioned, we fix an integer $\ell \geq 1$ and a matrix $\Gamma \in \cM^*_{\ell,K}(\Z)$. For any $\bg \in \Z^\ell$, any $\btheta \in \R^K$, any $i \in [K]$, define $\cT^{\btheta,(i)}_{\Gamma, \bg}$ as the tree $\cT^{\btheta,(i)}$ conditioned on $\Gamma \ \bN(\cT^{\btheta,(i)}) = \bg$.
The interest of the Definition \ref{def:eqtilt} lies in the following result. 

\begin{Proposition}
Fix $i \in [K],\ \btheta,\btheta' \in \R^K$ such that $\btheta \sim_\Gamma \btheta'$, and $\bg \in \Z^\ell$ such that \\ $\Prob(\Gamma \ \bN(\cT^{\btheta,(i)})=\bg)>0$. Then, the following holds:
\begin{itemize}
    \item[(i)] We have $\Prob(\Gamma \ \bN(\cT^{\btheta',(i)})=\bg)>0$.
    \item[(ii)] $\cT^{\btheta,(i)}_{\Gamma, \bg} \overset{(d)}{=} \cT^{\btheta',(i)}_{\Gamma, \bg}$.
\end{itemize}
\end{Proposition}

\begin{proof}
    The proof of (i) is clear (and it actually holds for any $\btheta' \in \R^K$), while (ii) is a consequence of \cite[Proposition $11$]{The23}.
\end{proof}

\subsection{Main results}
\label{ssec:main results}

Our main result is a dichotomic description of the structure of the cone of asymptotic directions, defined as follows:
\begin{Definition}[Asymptotic cone]
We define the asymptotic cone
  $$\mathcal{D}_{asy}:=\{\lambda X_{\btheta}|\lambda\in(0,+\infty),\btheta\in\mathcal{M}_{\text{crit}}\},$$ 
as the set of all asymptotic directions (up to a multiplicative constant). 
Furthermore, for any $\overline{\btheta}\in\R^K$ and any $\Gamma\in \cM^*_{\ell,K}(\Z)$ of rank $\ell$, we define the $(\Gamma,\overline{\btheta})$-asymptotic cone as:
\[
\mathcal{D}^{\Gamma,\overline{\btheta}}_{asy}:=\{\lambda \Gamma X_{\btheta} |\lambda\in(0,+\infty),\btheta\in\mathcal{M}_{\text{crit}}, \btheta \sim_\Gamma \overline{\btheta}\}.
\]
\end{Definition}

We can characterize the critical tiltings that are $\Gamma$-equivalent to a given $\overline{\btheta}$. This is done in the following theorem.
\begin{Theorem}
\label{thm:accessdir2}
Assume that $\bmu$ is entire, finite, nondegenerate, nonlocalized and irreducible.
Then, the set $\mathcal{D}_{asy}$ is nonempty, open and convex.\\
Furthermore, for any $\ell \geq 1$, any $\Gamma\in\cM^*_{\ell,K}(\Z)$ and any $\overline{\btheta}\in\R^K$, either $\mathcal{D}_{asy}^{\Gamma,\overline{\btheta}}=\{\bzero\}$ or $\bzero \notin \mathcal{D}_{asy}^{\Gamma,\overline{\btheta}}$.
\begin{itemize}
\item In the first case, $\overline{\btheta}$ is critical, $\Gamma X_{\overline{\btheta}}=\bzero$ and $\{\btheta \ | \ \btheta \sim_\Gamma \overline{\btheta}\}=\{\overline{\btheta}\}$.
\item In the second case, for any $X\in\mathcal{D}_{asy}$ such that $\Gamma X\not = 0$, there exists a unique couple $(\btheta,\lambda) \in \mathcal{M}_{\text{crit}} \times (0,+\infty)$ such that $\btheta \sim_\Gamma \overline{\btheta}$ and $\Gamma X_{\btheta}=\lambda \Gamma X$. This $\btheta$ is the unique maximizer of $f_X:\btheta\mapsto -X^{\intercal} \chi(\btheta)$ on $\{\btheta \ | \ \btheta \sim_\Gamma \overline{\btheta}\}$.
\end{itemize}
\end{Theorem}

The proof of this result is based on Theorem \ref{thm:accessdir}, where we also give a different characterization of the set $\mathcal{D}_{asy}$. 

As a corollary, we obtain the following result.

\begin{Corollary}\label{cor:coefsinz}
Assume that $\bmu$ is entire, finite, nondegenerate, nonlocalized and irreducible. Then for any $\Gamma\in\cM^*_{\ell,K}(\Z)$, there exists $\btheta \sim_\Gamma \bzero$ such that $\bmu_{\btheta}$ is critical.
\end{Corollary}

Observe in particular that the matrix $\Gamma$ may have negative coefficients.

\emph{Idea of the proof}

The strategy of the proof of Theorem \ref{thm:accessdir2} is to study the properties of the set $\cM_{crit}$ of all critical exponential tiltings of $\bmu$ - and not only those that are $\Gamma$-equivalent to $\bmu$. This set admits an elegant parametrization by the function $\chi$ defined in Definition \ref{def:functionc}; the image of $\cM_{crit}$ under $\chi$ is the boundary of the convex set $\mathcal{C}_{image}$, whose study is central in our proofs.

\begin{Remark}
We restrict ourselves to the case where $\bmu$ is entire. See Appendix \ref{sec:appendix} for an example of what may go wrong if it is not the case, even in the case where $\bmu$ is supercritical.
\end{Remark}

As a corollary of Theorem \ref{thm:accessdir2}, we obtain the following local limit result for size-conditioned trees. We require for this an additional \textit{aperiodicity} condition.

\begin{Definition}
\label{def:aperiodic}
    We say that a multitype projection $\bmu$ with support $Supp(\bmu)$ is \textit{aperiodic} if, for each $x \in \Z^K$, the smallest subgroup of $\Z^K$ containing $x+Supp(\bmu)$ is $\Z^K$.
\end{Definition}

\begin{Theorem}
\label{thm:critical local limit}
Let $\bze$ be a critical, aperiodic, finite, nondegenerate, nonlocalized and irreducible distribution, and set $\Gamma \in \cM^*_{\ell,K}(\Z)$ of rank $\ell$. Assume that either 
       \begin{itemize}
           \item  $\bze$ is not critical ;
           \item or $\bze$ is critical and $\Gamma X_{\bze} \neq \bzero$, where $X_{\bze}$ is the normalized $1$-left eigenvector of $M$.
       \end{itemize}
Let also $\alpha := (\alpha_1, \ldots, \alpha_\ell) \in \Gamma( \mathcal{D}_{asy})$ such that $\alpha \neq \bzero$. Let $(\bk(n), n \geq 1)$ be a sequence of elements of $\Z_+^\ell$ satisfying 
    \begin{itemize}
        \item $||\bk(n)||_1 \rightarrow \infty$;
       \item for all $i \in [K]$, $\frac{k_i(n)}{||\bk(n)||_1} \underset{n \rightarrow \infty}{\rightarrow} \frac{\alpha_i}{||\alpha||_1}$;
        \item $\Prob(\Gamma \ \bN(\cT^{(1)})=\bk(n)^\intercal)>0$.
    \end{itemize}
    Then, we have:
    \begin{align*}
        \cT^{(1)}_{\Gamma, \bk(n)} \underset{n \rightarrow \infty}{\overset{(d),loc}{\rightarrow}} \cT_{*}^{(1)}
    \end{align*}
   where $\cT_{*}^{(1)}$ is the multitype Kesten tree associated to some tilted critical measure $\tbze$.
\end{Theorem}

\noindent This theorem is proved in Section \ref{sec:locscallim}.

\section{Properties of the set $\cM_{crit}$ of critical tiltings}
\label{sec:properties of Mcrit}

The aim of this section is to investigate the set $\cM_{crit}$ of critical tiltings of a projection $\bmu$, that is, the set
$$
\cM_{crit} := \left\{ \btheta \in \R^K, \rho_\btheta=1 \right\}.
$$

Note that, a matrix $\Gamma$ being given, the elements of $\cM_{crit}$ are not necessarily $\Gamma$-equivalent to $\bzero$.

\begin{Remark}[Results on $\chi$, $\mathcal{C}_{image}$ and $\cM_{crit}$]
The precise links between $\mathcal{C}_{image}$ and $\cM_{crit}$ may not appear completely obvious at first glance. Figure \ref{fig:cimage} represents an example of $\mathcal{C}_{image}$. It turns out that it is always a closed and convex set (see Proposition \ref{prop:convex}) and, in addition, on the boundary the normal vector has positive coefficients. The most important property of $\mathcal{C}_{image}$ is that its boundary is in correspondence with $\cM_{crit}$ in the sense that $\chi$ restricted to $\cM_{crit}$ is a bijection between $\cM_{crit}$ and the boundary of $\mathcal{C}_{image}$ (see Proposition \ref{prop:bijection}). Furthermore, for any $\btheta\in\cM_{crit}$, the asymptotic direction $X_{\btheta}$ is (up to a multiplicative constant) the normal vector of the boundary of $\mathcal{C}_{image}$ at $\chi(\btheta)$.

\begin{figure}
\begin{center}
\begin{tikzpicture}[scale=0.5]

\fill[pattern=horizontal lines, pattern color=blue!20,samples=200] plot[domain=-5:-0.6945] ({2*\x-2*ln(1-2*exp(\x))+ln(1-exp(\x))-ln(3)-ln(exp(\x)+1},{ln(1-2*exp(2*\x)*(1+exp(\x)))-ln(3)-2*\x}) -- (7.5,9) -- cycle;

\draw[thick,color=blue] plot[domain=-5:-0.694,samples=300] ({2*\x-2*ln(1-2*exp(\x))+ln(1-exp(\x))-ln(3)-ln(exp(\x)+1},{ln(1-2*exp(2*\x)*(1+exp(\x)))-ln(3)-2*\x});

\draw[->] (-14,0) -- (10,0) node[anchor=west] {$\chi_1(\btheta)$};
\draw[->] (0,-2.5) -- (0,10) node[anchor=south] {$\chi_2(\btheta)$};
\end{tikzpicture}
\caption{ In dashed blue, the set $\mathcal{C}_{image}$ for the projection $\bmu$ with generating functions \\ $\phi^{(1)}(x_1,x_2)  = \frac{1}{3}\left( x_1x_2^2 + x_1x_2 + x_2 \right)$ and $\phi^{(2)}(x_1,x_2) = \frac{1}{3}\left( x_1x_2+x_2+1 \right)$. \qquad \qquad \ \ \ \ \textcolor{white}{Non}  \\ The boundary of $\mathcal{C}_{image}$, in blue, has parametrization: \ \ \ \ \ \qquad \qquad \qquad \qquad \qquad \qquad \textcolor{white}{Oui}\\ $\Big(\big(2s-2\ln(1-2e^s)+\ln(1-e^s)-\ln(3)-\ln(1+e^s)\big),\big(\ln(1-2e^{2s}(1+e^s))-\ln(3)-2s\big)\Big)$,\\
for $s\in(-\infty,-\ln(2))$. \ \ \ \ \qquad \qquad \qquad \qquad \qquad \qquad \qquad \qquad \qquad \qquad \qquad \qquad\textcolor{white}{Peut-être}}
\label{fig:cimage}
\end{center}
\end{figure}
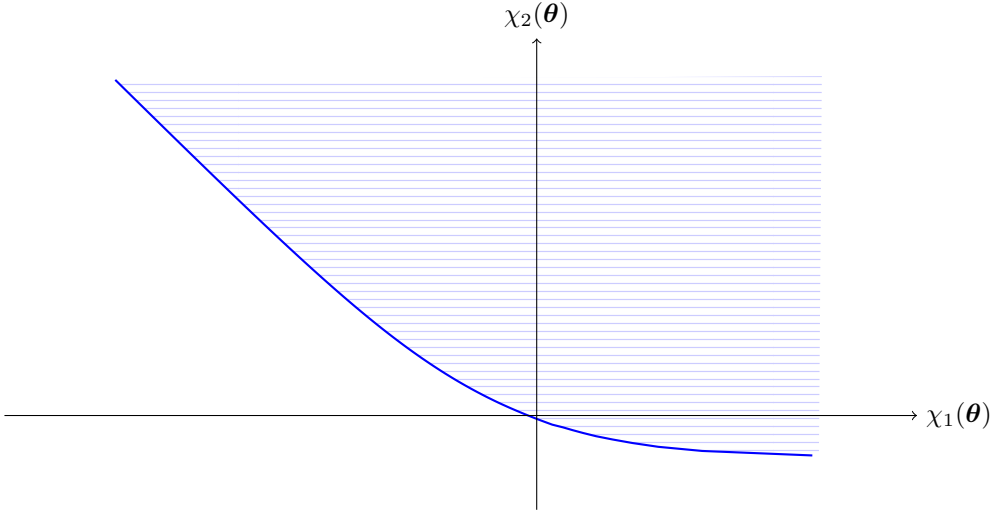
\end{Remark}

\subsection{Critical tiltings as maximizers}

We start by introducing a function $f_X$, which allows us to charaterize critical tiltings.

\begin{Definition}
\label{def:fx}
For any vector $X\in(0,+\infty)^K$ with positive coefficients, we define the function $f_X:\R^K\rightarrow \R$ as:
\[
f_X(\btheta):=\sum_i X_i \big(\theta_i-\log\big(\phi^{(i)}\big(e^{\btheta}\big)\big)\big)=-X^{\intercal}\chi(\btheta).
\]
\end{Definition}

We will show a link between, on the one hand, critical parameters and the associated asymptotic direction and, on the other hand, the functions $f_X$ and their maximizers.
This will also show that, for any $\btheta$ such that $\bmu_{\btheta}$ is critical, $\chi(\btheta)$ lies in the boundary of $\mathcal{C}_{image}$ and that the asymptotic direction $X_{\btheta}$ is normal to the boundary of $\mathcal{C}_{image}$ at $\chi(\btheta)$.
\begin{Lemma}
\label{lem:resultonc1}
Let $\bmu$ be entire and let $X\in(0,+\infty)^K$ be a vector with positive coefficients. Then, the following holds:
\begin{itemize}
    \item[(i)] The function $f_{X}$ is concave. Furthermore, if $\bmu$ is nonlocalized, then $f_{X}$ is strictly concave (in the sense that its Hessian matrix is negative definite).
\item[(ii)]
Take $\btheta\in\R^K$, and assume that $\bmu_{\btheta}$ is critical. Let $X_{\btheta}$ be the asymptotic direction associated to $\bmu_\btheta$. Then, $f_{X_{\btheta}}$ is maximal at $\btheta$. Furthermore, if $\bmu$ iis the s nonlocalized, this maximizer is unique.
\end{itemize}
\end{Lemma}

\begin{proof}[Proof of Lemma \ref{lem:resultonc1}]
To simplify notations, we will write for any $i\in[K]$, $g^{(i)}(\btheta):=\log(\phi^{(i)}(e^{\btheta}))$. 
To prove (i), it suffices to compute the Hessian $H_{\btheta}$ of the function $f_X$ at $\btheta \in \R^K$. We have:
\[
\forall i,j \in [K], H_{\btheta}(i,j)= -\sum_{k\in [K]} X_k \frac{\partial^2}{\partial \theta_i \partial \theta_j} g^{(k)} (\btheta).
\]
In particular, for any vector $Z:=(Z_1, \ldots, Z_K)\in\R^K$, we have:
\[
Z^{\intercal}H_{\btheta}Z= - \sum_{k\in [K]} X_k \sum_{1\leq i,j \leq K} Z_i Z_j \frac{\partial^2}{\partial \theta_i \partial \theta_j} g^{(k)}(\btheta).
\]
Notice that $V_k:=\sum_{1\leq i,j \leq K} Z_i Z_j \frac{\partial^2}{\partial \theta_i \partial \theta_j} g^{(k)}(\btheta)$ is just the variance of $Z\cdot \mathbf{L}$ under $\mu^{(k)}_{\btheta}$ where $\mathbf{L}:=(L_1,\dots,L_K)$ is the number of children of each type under $\mu_{\btheta}^{(k)}$.

By convexity of $g^{(k)}$, $V_k\geq 0$ for all $k$, so $f_X$ is always concave. In addition, when $\bmu$ (and therefore $\bmu_{\btheta}$) is nonlocalized, there exists $k\in[K]$ such that $V_k>0$. Hence, the Hessian matrix of $f_X$ is always negative definite and thus $f_X$ is strictly concave.\\

Let us now prove (ii). For any $\btheta,Y \in \R^K$, all $X \in (0,+\infty)^K$, we have:
\begin{equation}
\label{eq:maximizer}
\sum_{i \in [K]} Y_i\frac{\partial}{\partial \theta_i } f_{X}(\btheta) 
=\sum_{i,j \in [K]} Y_i X_j\left(1_{i=j}-\frac{\partial}{\partial \theta_i}g^{(j)}(\btheta)\right)
=X^{\intercal}(I_K-M_{\btheta})Y
\end{equation}
Assume that $\btheta$ is critical. Specifying \eqref{eq:maximizer} at $X_\btheta$, we get that the gradient of $f_{X_{\btheta}}$ at $\btheta$ is $\bzero$. Therefore, $\btheta$ is a global maximizer of $f_X$ as the function is concave. Furthermore, if $\bmu$ is nonlocalized, the function $f_{X_{\btheta}}$ is strictly concave and $\btheta$ is therefore the unique maximizer. 
\end{proof}

\subsection{$\cM_{crit}$ is non-empty}
\label{ssec:nonempty}

We now prove that the set $\cM_{crit}$ of critical tiltings is nonempty.

\begin{Proposition}
\label{prop:Mcrit is nonempty}
For any projection $\bmu$ which is entire, finite, nondegenerate and irreducible, the set $\cM_{crit}$ is nonempty.
\end{Proposition}

To prove that there exists $\btheta \in \R^K$ such that $\bmu_\btheta$ is critical, by continuity of $\btheta \mapsto \rho_\btheta$ it is enough to prove that there exist $\btheta_1, \btheta_2$ such that $\bmu_{\btheta_1}$ is subcritical and $\bmu_{\btheta_2}$ is supercritical. It turns out that we can prove a stronger result, namely, that for any $\bmu$ finite, one can always find $\btheta \in \R^K$ such that the measure $\bmu_\btheta$ is subcritical is $\Gamma$-equivalent to $\bmu$, jointly for all $\Gamma \in \cup_{\ell \geq 1} \cM^*_{\ell,K}(\Z)$. In particular this holds for any finite supercritical $\bmu$.
\begin{Lemma}
\label{lem:fromsupercrittosubcrit}
Let $\bmu$ be finite. For $i \in [K]$, let $p_i$ be the probability under $\bmu$ that the tree $\cT^{(i)}$ is finite. Since $\bmu$ is finite, $p_i \in (0,1]$ for all $i \in [K]$. Define $\mathbf{q} := (\log (p_1), \ldots, \log (p_K))$. Then the measure $\bmu_{\mathbf{q}}$ is subcritical. Furthermore, $\chi(\bzero)=\chi(\mathbf{q})$, so that $\bmu_{\mathbf{q}}$ is $\Gamma$-equivalent to $\bmu$ for any $\Gamma$. If one assumes in addition that $\bmu$ is nonlocalized, then the measure $\bmu_{\mathbf{q}}$ is critical if and only if $\bmu$ is critical, in which case $\mathbf{q}=\bzero$.
\end{Lemma}

\begin{proof}
Splitting according to the children of the root of $\cT^{(i)}$, we get that
\begin{align*}
p_i &= \sum_{(k_1, \ldots, k_K) \in \Z_+^K} \mu^{(i)}(k_1, \ldots, k_K) \prod_{j=1}^K p_j^{k_j}\\
&= \phi^{(i)}\left(p_1, \ldots, p_K \right).
\end{align*}
Hence, we have $\chi(\mathbf{q}) = \bzero = \chi(\bzero) $ (recall Definition \ref{def:eqtilt}), and it follows that $\bmu_{\mathbf{q}} \sim_\Gamma \bmu$ for all $\Gamma$. 

We now prove that $\bmu_{\mathbf{q}}$ is subcritical. By \cite[Theorem $7.1$]{Har63}, it is equivalent to showing that $r_i=1$ for all $i \in [K]$, where $r_i = \Prob(\cT^{\mathbf{q},(i)} \text{ is finite})$. Observe that, by definition, $\mathbf{r} := (r_1, \ldots, r_K)$ is a solution of the equation
\begin{align*}
 \left\{
    \begin{array}{ccc}
        r_{1} & = & \phi_{\mathbf{q}}^{(1)}(\mathbf{r}), \\
        & \vdots & \\
        r_K & = & \phi_{\mathbf{q}}^{(K)}(\mathbf{r}).
    \end{array}
\right.
\end{align*}


Clearly, $(1, \ldots, 1)$ is a solution of this equation. Furthermore, by e.g. \cite[Corollary $1$ of Theorem $7.2$ of Chapter II]{Har63}, it is the only solution in the unit cube. 
Finally, we need to understand when $\bmu_{\mathbf{q}}$ is critical. By Lemma \ref{lem:resultonc1} (ii), if $\bmu$ is nonlocalized and $\mathbf{q}\not=\bzero$, then $\bmu_{\mathbf{q}}$ cannot be critical as $\bzero$ would be another maximizer of $f_{X_{\mathbf{q}}}$. The result follows.
\end{proof}

Let us immediately mention the following corollary, which can also be found in \cite{Pen16}.

\begin{Corollary}
\label{cor:fromsupercrittosubcrit}
    If $rk(\Gamma)=K$, then any supercritical and finite measure $\bmu$ has a critical $\Gamma$-equivalent tilting.
\end{Corollary}

Note that no assumption on the moments of $\bmu$ is needed.

\begin{proof}
Define $\mathbf{q}$ as in Lemma \ref{lem:fromsupercrittosubcrit}. By Lemma \ref{lem:fromsupercrittosubcrit} and continuity of $\btheta \mapsto \rho_\btheta$, there exists $\lambda \in [0,1]$ such that $\bmu_{\lambda \mathbf{q}}$ is critical (observe that $\bmu_{\lambda \mathbf{q}}$ is well defined, as $\lambda \in [0,1]$). To conclude, observe that, if $rk(\Gamma)=K$, then $(\ker \, \Gamma)^\perp = \R^K$, and thus for any $\btheta \in \R$, $\btheta \sim_\Gamma \bzero$. The result follows.
\end{proof}

\bigskip

In the case where $\bmu$ is subcritical, we need to be more careful. Indeed, even in the monotype case there are examples where $\bmu$ has no supercritical $\Gamma$-equivalent measure, and in particular no critical one, see e.g. \cite[Remark 4.3]{Jan12}. Hence, we need to assume here that $\bmu$ is entire.

Let $\textbf{1}=(1,\dots,1)$ and, for $s \in \R$, $s\mathbf{1} = (s, \ldots, s)$. We will show that, for any projection $\bmu$ which is entire, for $s$ large enough, $\bmu_{s\textbf{1}}$ is supercritical.

\begin{Lemma}\label{lem:supercritical}
Assume that $\bmu$ is entire, nondegenerate and irreducible. Then, for $s$ large enough, $\bmu_{s\mathbf{1}}$ is supercritical.
\end{Lemma}

This immediately implies Proposition \ref{prop:Mcrit is nonempty}.

\begin{proof}[Proof of Proposition \ref{prop:Mcrit is nonempty}]
Proposition \ref{prop:Mcrit is nonempty} is an immediate consequence of Lemmas \ref{lem:fromsupercrittosubcrit} and \ref{lem:supercritical}, as well as the continuity of $\btheta \mapsto \rho_{\btheta}$.    
\end{proof}

\begin{proof}[Proof of Lemma \ref{lem:supercritical}]
By assumption, $\bmu$ is entire, so that $\bmu_{s\mathbf{1}}$ is well-defined for all $s>0$.
Let $I$ be the set of types that can have at most one child and $J$ the set of types that can have at least two children. Since $\bmu$ is nondegenerate, $J$ is non-empty. For all $s>0$, all $i \in [K]$, denote by $Z_s(i)$ the probability under $\bmu_{s\mathbf{1}}$ that $\cT^{s \mathbf{1},(i)}$ contains a vertex whose type is in $J$. In particular, $Z_s(i)>0$ for all $i \in [K]$ by irreducibility of $\bmu$, and $Z_s(j)=1$ for all $j \in J$.
Our goal is to show that for $s$ large enough, for any type $k \in [K]$, 
\begin{equation}
\label{eq:larger}
    (M_{s\mathbf{1}} \, Z_s)(k)\geq Z_s(k).
\end{equation} 
By the Perron-Frobenius theorem, this will imply that we are in the supercritical case.

Let us first show that, for $s$ large enough, a vertex with type $j \in J$ has on average at least $3/2$ children under $\bmu_{s\mathbf{1}}$. For all $j\in J$ and all $n \geq 0$, let $q_{j,s}(n)$ be the probability under $\mu^{(j)}_{s\mathbf{1}}$ that a node of type $j$ has $n$ children. For all $j,s$ there exists a renormalizing constant $C_{j,s}$ such that:
\[
q_{j,s}(n)=C_{j,s}q_{j,0}(n)e^{sn}.
\]

 As a consequence, the probability that a vertex of type $j \in J$ has $0$ or $1$ child satisfies:

\[
\begin{aligned}
q_{j,s}(0)+q_{j,s}(1)
=&\frac{q_{j,0}(0)+q_{j,0}(1)e^{s}}{\sum_n q_{j,0}(n)e^{sn}}\\
\leq & \frac{\big(q_{j,0}(0)+q_{j,0}(1)\big)e^{s}}{\sum_{n\geq 2} q_{j,0}(n)e^{2s}}\\
=& \frac{q_{j,0}(0)+q_{j,0}(1)}{(1-q_{j,0}(0)-q_{j,0}(1))}e^{-s} \underset{s \rightarrow \infty}{\rightarrow} 0.
\end{aligned}
\]

Hence, for $s$ large enough, the average number of children of a vertex of type $j$ is at least $3/2$.

Let us now prove \eqref{eq:larger}.
Assume first that $I=\emptyset$. Then, we have
\begin{equation*}
    \forall j\in J, \sum_{k \in [K]} M_{s\mathbf{1}}(j,k) Z_s(k) = \sum_{k \in [K]} M_{s\mathbf{1}}(j,k) \geq 3/2 \geq 1=Z_j,
\end{equation*}
and the result follows.

Now assume that $I \neq \emptyset$.  As elements with type in $I$ have at most one child, we have the following equality:
\begin{equation}
\label{eq:type1}
\forall i\in I, Z_s(i)=\sum_{k \in [K]} M_{s\mathbf{1}}(i,k) Z_s(k).    
\end{equation}

So for all $i \in I$ and for all $s$, $(M_{s\mathbf{1}} Z_s(i))=Z_s(i)$ and \eqref{eq:larger} holds.

We have that for any $i\in I$ and any $j \in [K]$:
\begin{equation}
\label{eq:moyennetype1}
M_{s\mathbf{1}}(i,j)=\frac{e^s M(i,j)}{\mu^{(i)}(\bzero)
+e^s \sum_{k \in [K]} M(i,k)}.
\end{equation}

Let $(Z_\infty(1), \ldots, Z_\infty(K))$ be any subsequential limit of $(Z_s(1), \ldots, Z_s(K))$. By \eqref{eq:type1} and \eqref{eq:moyennetype1}, we obtain that, for all $i \in I$:
\begin{equation*}
    Z_\infty(i) = \frac{\sum_{k=1}^K M(i,k) Z_\infty(k)}{\sum_{k=1}^K M(i,k)}.
\end{equation*}
In particular, by irreducibility of $M$, we get that $Z_\infty(i)=1$ for all $i \in I$. Thus, $(Z_s(1), \ldots, Z_s(K)) \rightarrow \mathbf{1}$ as $s \rightarrow \infty$. In particular, for $s$ large enough, $\inf_{i \in I} Z_s(i) \geq 3/4$. We now use the following fact. Since $Z_s(j')=1$ for all $j' \in J$, if $I \neq \emptyset$, we have:
\begin{equation}
\label{eqn:ineqmanychild}
    \forall j\in J, \sum_{k \in [K]} M_{s\mathbf{1}}(j,k) Z_s(k) \geq \big(\inf_{i\in I} Z_s(i)\big)\sum_{k \in [K]} M_{s\mathbf{1}}(j,k)
\end{equation}

Therefore, for $s$ large enough, for all $j \in J$, $(M_{s \mathbf{1}} Z_s)(j) \geq 3/4 \times 3/2 \geq 1 = Z_s(j)$ and \eqref{eq:larger} also holds for any $j \in J$. Hence, we have that $(M_{s\mathbf{1}}Z_s)(i)\geq Z_s(i)$ for all $i \in [K]$. Moreover, $M_{s\mathbf{1}}$ has nonnegative coefficients. Thus, we get that for all $n\in\N, \big(M_{s\mathbf{1}}^n Z_s \big)(i)\geq \big(M_{s\mathbf{1}} Z_s\big)(i)>0 $. This implies that $\rho_{s\mathbf{1}} \geq 1$ and therefore $\bmu_{s \mathbf{1}}$ is supercritical. 

\end{proof}

\subsection{$\cM_{crit}$ is a manifold}
\label{ssec:open}

We prove here that the set $\cM_{crit}$ has a manifold structure.

\begin{Proposition}\label{prop:open}
Let $\bmu$ be entire, finite, nondegenerate, nonlocalized and irreducible. Then, the set $\cM_{crit}$ is a differentiable manifold of dimension $K-1$.
\end{Proposition}

This is shown by exhibiting a local diffeomorphism (see Lemma \ref{lem:invertible}) around each element of $\cM_{crit}$. The next lemma allows us to simplify the study of critical matrices.

\begin{Lemma}
\label{lem:nul}
Let $\bmu$ be entire and irreducible. Let $\btheta \in \cM_{crit}$. Then, there exists $\varepsilon>0$ such that $$\cM_{crit} \cap B_\varepsilon(\btheta) = \left\{\btheta', \det(M_{\btheta'}-I_K)=0 \right\} \cap B_\varepsilon(\btheta).$$
\end{Lemma}

\begin{proof}
Since $\bmu$ is entire, $\bmu_\btheta$ is well-defined for all $\btheta \in \R^K$. Fix $\btheta \in \cM_{crit}$. Assume that there exists a sequence $(\btheta^{(n)}, n \geq 1)$ converging to $\btheta$, such that, for all $n$, $\det(M_{\btheta^{(n)}}-I_K)=0$ but $\bmu_{\btheta^{(n)}}$ is noncritical. Since $1$ is an eigenvalue of $M_{\btheta^{(n)}}$, necessarily $\bmu_{\btheta^{(n)}}$ is supercritical (but noncritical). In particular, by the Perron-Frobenius theorem, for all $n$ the largest eigenvalue of $M_{\btheta^{(n)}}$ is a real number $z_n>1$. Up to extracting a subsequence, we can assume that $z_n$ converges to some $z \geq 1$. Indeed, $z_n \leq \sum_{i,j} \left|\left( M_{\btheta^{(n)}} \right)_{i,j} \right|$ which is a bounded continuous function of $\btheta$, so that $(z_n)_{n \geq 1}$ is bounded. In particular $z$ is an eigenvalue of $M_{\btheta}$ by continuity of the determinant. If $z>1$, then $\bmu_{\btheta}$ is noncritical. On the other hand, $z=1$ cannot occur since, in that case, $1$ would be an eigenvalue of multiplicity at least 2 of $M_{\btheta}$. This is impossible by the Perron-Frobenius theorem. The result follows.
\end{proof}

Now let $\pi$ be a linear map from $\R^K$ to $\R^{K-1}$ of kernel $\R \bone$.
We define the function $F: \R^K\rightarrow \R^K$ as follows:
\[
F(\btheta):=\left(\pi(\chi(\btheta)),\det(M_{\btheta}-I_K)\right),
\]
where the first $K-1$ coordinates of $F$ are $\pi(\chi(\btheta))$.

\begin{Lemma}
\label{lem:invertible}
Assume that $\bmu$ is entire, nonlocalized and irreducible. The function $F$ is a local diffeomorphism around any $\btheta \in \R^K$ such that $\bmu_{\btheta}$ is critical.  
\end{Lemma}

Clearly, Lemmas \ref{lem:nul} and \ref{lem:invertible} along with Proposition \ref{prop:Mcrit is nonempty} imply Proposition \ref{prop:open}.

\begin{proof}[Proof of Lemma \ref{lem:invertible}]
Fix $\btheta$ such that $\bmu_{\btheta}$ is critical. We will show that $F$ is a local diffeomorphism around $\btheta$, which boils down to proving that its Jacobian matrix at $\btheta$ is invertible. First, observe that the Jacobian matrix $J_{\chi}(\btheta)$ of the function $\chi$ at $\btheta$ satisfies $J_{\chi}(\btheta)=M_{\btheta}-I_K$. Let us prove that $\ker(\pi \circ J_{\chi}(\btheta))=\ker(J_{\chi}(\btheta))$ (seeing $J_{\chi}(\btheta)$ as a map from $\R^K$ to $\R^K$). This is equivalent to showing that $\text{Im}(J_{\chi}(\btheta))\cap \ker(\pi) = \{ \bzero \}$.
Let $Z \in \text{Im}(J_{\chi}(\btheta))\cap \ker(\pi)$ and assume that $Z \neq \bzero$. By definition of $\pi$, $Z \in \R \bone$, so that $Z$ has only negative coordinates or only positive ones. On the other hand, we have $\text{Im}(J_{\chi}(\btheta))=\ker(J_{\chi}(\btheta)^{\intercal})^{\perp}$. By the Perron-Frobenius theorem, $\ker(J_{\chi}(\btheta)^{\intercal})$ is a one-dimensional vector space generated by the vector $X_{\btheta}$ which has positive coordinates (see Definition \ref{def:asymdir}). Since $Z \in \text{Im}(J_{\chi}(\btheta))$, we would have $X_{\btheta} \cdot Z=0$ which is not possible because $Z$ has coordinates of the same sign and $Z \neq \bzero$. Hence, $\ker(J_{\chi}(\btheta))=\ker(\pi \circ J_{\chi}(\btheta))$.

By Jacobi's formula, we have:
\begin{equation} 
\label{eq:jacobi}
    \frac{\partial}{\partial x_i} \det(M_{\btheta}-I_K) = \text{Tr}\left(\text{com}(M_{\btheta}-I_K)^{\intercal}\frac{\partial}{\partial x_i} M_{\btheta}\right),
\end{equation}
where $\text{com}(A)$ denotes the cofactor matrix of $A$. Since $\bmu_{\btheta}$ is critical, $M_{\btheta}-I_K$ is of rank $K-1$ and $\text{com}(M_{\btheta}-I_K)$ is of rank $1$. Therefore, there exists $\lambda\in\R\backslash\{0\}$ such that  $\text{com}(M_{\btheta}-I_K)^{\intercal}=  \lambda  Y X_{\btheta}^{\intercal}$, where $Y \in (0,+\infty)^K$ is a 1-right eigenvector of $J_{\chi}(\btheta)$.  (remember that $X_{\btheta}^{\intercal}M_{\btheta} = X_{\btheta}^{\intercal}$ and $M_{\btheta} Y= Y$). 

We now have all the tools to prove that the Jacobian matrix $J_F$ of $F$ is invertible at $\btheta$, that is, $\ker(J_F(\btheta))=\{ \bzero \}$. To this end, let $U \in \ker(J_F(\btheta))$ and assume by contradiction that $U \neq  \bzero $. By definition, we have $U \in \ker \pi(J_{\chi}(\btheta)) = \ker(J_{\chi}(\btheta))$, so $U=\nu Y$ for some $\nu \in \R \backslash \{0\}$. In addition, if $U \in \ker(J_F)$, we have by \eqref{eq:jacobi}:
\begin{equation*}
    \sum_i U_i \text{Tr}\left(\text{com}(M_{\btheta}-I_K)^{\intercal}\frac{\partial}{\partial x_i} M_{\btheta}\right)=0.
\end{equation*}
We use again the notation $g^{(i)}(\btheta):=\log(\phi^{(i)}(e^{\btheta}))$, so that in particular $M_{\btheta}(i,j)= \frac{\partial}{\partial x_j} g^{(i)}(\btheta)$. Computing the left-hand term, we get
\[
\begin{aligned}
\sum_i U_i \text{Tr}\left(\text{com}(M_{\btheta}-I_K)^{\intercal}\frac{\partial}{\partial x_i} M_{\btheta}\right)
=&  \sum_i \nu Y(i) \text{Tr}\left( \lambda Y X_{\btheta}^{\intercal}\frac{\partial}{\partial x_i} M_{\btheta}\right)\\
=& \lambda \nu\sum_i Y(i) \sum_{j,k}  X_{\btheta}(j)  \left(\frac{\partial}{\partial x_i} M_{\btheta}\right)_{j,k}Y(k)\\
=& \lambda \nu \sum_i Y(i) \sum_{j,k} X_{\btheta}(j) Y(k) \frac{\partial^2}{\partial x_i \partial x_k}\big(g^{(j)}(\btheta)\big)\\
=& \lambda \nu \sum_j X_{\btheta}(j) \sum_{i,k} Y(i)Y(k) \frac{\partial^2}{\partial x_i \partial x_k}\big(g^{(j)}(\btheta)\big)\\
=& \lambda \nu \sum_{i,j} Y(i) Y(j) \left( H(-f_{X_{\btheta}}) \right)_{i,j},
\end{aligned}
\]

where $H(-f_{X_{\btheta}})$ is the Hessian of the function $-f_{X_{\btheta}}$, where $f_X$ is defined in Definition \ref{def:fx}.
Since $H(-f_{X_{\btheta}})$ is positive definite by Lemma \ref{lem:resultonc1} (i) (using the fact that $\bmu$ is nonlocalized) and $Y \neq\bzero$, the right-hand side cannot be equal to $0$. This implies that $\ker(J_F)=\{\bzero\}$, so $J_F$ is invertible at $\btheta$.
\end{proof}

\section{Properties of the function $\chi$}
\label{sec:propertiesofthefunctionc}

In this section, we show some properties of the function $\chi$, which will be useful in the proof of our main results.

\subsection{Structure of the asymptotic cones}

Here we prove a general result about asymptotic directions, which may be of independent interest - namely, the asymptotic cone $\mathcal{D}^{\Gamma,\overline{\btheta}}_{asy}$ (see Definition \ref{def:asymdir})  is open whenever it does not contain $\bzero$. More generally, the structure of the asymptotic cone strongly depends on whether $\bzero \in\mathcal{D}^{\Gamma,\overline{\btheta}}_{asy}$ or not. The following lemma allows us to characterize the former case precisely.

\begin{Lemma}\label{lem:0inDasy}
    Assume that $\bmu$ is entire, finite, nondegenerate, nonlocalized and irreducible. Let $\overline{\btheta} \in \R^K$ and $\Gamma\in \cM^*_{\ell,K}(\Z)$ of rank $\ell$. The following are equivalent:
    \begin{itemize}
        \item[(i)] $\bzero\in \mathcal{D}^{\Gamma,\overline{\btheta}}_{asy}$ ;
        \item[(ii)] $\mathcal{D}^{\Gamma,\overline{\btheta}}_{asy}=\{\bzero\}$ ;
        \item[(iii)] $\overline{\btheta}$ is critical, $\Gamma X_{\overline{\btheta}}=\bzero$ and $\left(\btheta\sim_{\Gamma} \overline{\btheta} \implies \btheta= \overline{\btheta}\right)$.
    \end{itemize}
\end{Lemma}
\begin{proof}
Assume that $\overline{\btheta}$ is critical, $\Gamma X_{\overline{\btheta}}=\bzero$ and $(\btheta\sim_{\Gamma} \overline{\btheta} \implies \btheta= \overline{\btheta})$. Then, by definition of the asymptotic cone, $\mathcal{D}^{\Gamma,\overline{\btheta}}_{asy}=\{\bzero\}$, so (iii) $\Rightarrow$ (ii).\\
It is also clear that (ii) $\Rightarrow$ (i).\\
We only have to prove that (i) $\Rightarrow$ (iii). Assume that $\bzero\in\mathcal{D}^{\Gamma,\overline{\btheta}}_{asy}$. This means that there exists $\btheta'$ critical such that $\btheta'\sim_{\Gamma} \overline{\btheta}$ and $\Gamma X_{\btheta'}=\bzero$. By definition of $\Gamma$-equivalence, for any $\btheta\sim_{\Gamma}\overline{\btheta}$, we have $\btheta \sim_\Gamma \btheta'$ and thus there exists a vector $R \in \R^\ell$ such that $\chi(\btheta')-\chi(\btheta)=\Gamma^{\intercal}R$. This in turn means that the following holds:
\begin{equation}
    \forall \btheta\sim_{\Gamma}\overline{\btheta},\exists R\in\R^{\ell},\   X_{\btheta'}^{\intercal}(\chi(\btheta')-\chi(\btheta))
    =X_{\btheta'}^{\intercal}\Gamma^{\intercal}R=\bzero^{\intercal}R=0.
\end{equation}
On the other hand, using Lemma \ref{lem:resultonc1} (ii) and the fact that $\bmu$ is nonlocalized, we get 
that for any $\btheta\not=\btheta'$, we have:
\begin{equation}
    X_{\btheta'}^{\intercal}(\chi(\btheta')-\chi(\btheta))
    <0.
\end{equation}
This implies that, necessarily, $\btheta\sim_{\Gamma} \overline{\btheta} \implies \btheta= \btheta'$, which implies (iii).
\end{proof}

We now consider the second case, where $\bzero\not\in\mathcal{D}^{\Gamma,\overline{\btheta}}_{asy}$. In this case, it turns out that the set $\mathcal{D}^{\Gamma,\overline{\btheta}}_{asy}$ is open.

\begin{Proposition}\label{prop:opendirection2}
Assume that $\bmu$ is entire, finite, nondegenerate, nonlocalized and irreducible. Let $\overline{\btheta} \in \R^K$ and $\Gamma\in \cM^*_{\ell,K}(\Z)$ of rank $\ell$. Then, either $\mathcal{D}^{\Gamma,\overline{\btheta}}_{asy}=\{\bzero\}$ or the set $\mathcal{D}^{\Gamma,\overline{\btheta}}_{asy}$ is open and $\bzero\not\in \mathcal{D}^{\Gamma,\overline{\btheta}}_{asy}$. In particular, the set $\mathcal{D}_{asy}$ is open.
\end{Proposition}

The proof of Proposition \ref{prop:opendirection2} uses the following lemma, which states that the measures associated to two different critical parameters $\btheta\not=\btheta'$ cannot have the same asymptotic direction. For any $\btheta \in \cM_{crit}$, define
\begin{align*}
    E^{\Gamma,\btheta} := \{\btheta' \in \cM_{crit} | \btheta' \sim_\Gamma \btheta, \Gamma X_\btheta' \text{ and } \Gamma X_\btheta \text{ are collinear} \}.
\end{align*}

\begin{Lemma}\label{lem:injective2}
Let $\bmu$ be an entire, nonlocalized and irreducible measure, and let $\Gamma \in \cM^*_{\ell,K}(\Z)$ for some $\ell \geq 1$. Let $\btheta \in \mathcal{M}_{\text{crit}}$. Then the following holds:
\begin{itemize}
    \item[(i)] $|E^{\Gamma,\btheta}| \in \{1, 2\}$.
    \item[(ii)] If $|E^{\Gamma,\btheta}| = 2$ and $E^{\Gamma,\btheta} = \{\btheta,\btheta'\}$ then there exists $\lambda<0$ such that $\Gamma X_\btheta' =\lambda \Gamma X_\btheta$.
    \item[(iii)] In addition, if $\Gamma \in \cM^*_{\ell,K}(\Z_+)$, then $E^{\Gamma,\btheta} = \{ \btheta \}$.
\end{itemize}
\end{Lemma}

An immediate corollary is the following.

\begin{Corollary}
\label{cor:asymp dir are not collin}
    For any $\btheta, \btheta' \in \cM_{crit}$ such that $\btheta \neq \btheta'$, $X_\btheta\not= X_{\btheta'}$.
\end{Corollary}

\begin{proof}[Proof of Corollary \ref{cor:asymp dir are not collin}]
By Lemma \ref{lem:injective2} (iii) to $\Gamma = I_K \in \cM^*_{K,K}(\Z_+)$, $X_\btheta$ and $X_{\btheta'}$ are not collinear. Since $||X_\btheta||_1=||X_{\btheta'}||_1=1$ and both have positive coordinates, it is equivalent to saying that $X_\btheta \neq X_{\btheta'}$.
\end{proof}

\begin{proof}[Proof of Lemma \ref{lem:injective2}]
Fix $\btheta, \btheta' \in \cM_{crit}$ such that $\btheta \neq \btheta'$ and $\btheta \sim_\Gamma \btheta'$, and assume that $\Gamma X_{\btheta}$ and $\Gamma X_{\btheta'}$ are collinear. Without loss of generality, we can assume that there exists $\lambda \in \R$ such that $\Gamma X_{\btheta}=\lambda \Gamma X_{\btheta'}$. By Lemma \ref{lem:resultonc1} (ii), we have:
\[
X_{\btheta}^{\intercal}\big(\chi(\btheta')-\chi(\btheta)\big)>0.
\]
Similarly we have:
\[
X^{\intercal}_{\btheta'}\big(\chi(\btheta')-\chi(\btheta)\big)<0.
\]

By definition of $\Gamma$-equivalence (using that for any matrix $A$, $\ker(A)^{\perp}=\text{Im}(A^{\intercal})$), there exists $R\in\R^\ell$ such that $\chi(\btheta')-\chi(\btheta)=\Gamma^{\intercal} R$. We can thus write 
\begin{equation}
\label{eq:impossible si lambda > 0}
X_{\btheta'}^{\intercal}\Gamma^{\intercal} R<0< X_{\btheta}^{\intercal}\Gamma^{\intercal} R=\lambda X_{\btheta'}^{\intercal}\Gamma^{\intercal} R.
\end{equation}

\noindent In particular, if $\lambda \geq 0$ then there is no solution.

Let us prove (i) and (ii). Assume that $|E^{\Gamma,\btheta}|\geq 3$ and choose $\btheta_1, \btheta_2, \btheta_3 \in E^{\Gamma,\btheta}$ distinct. Then, necessarily there exists $i \neq j$ and $\lambda' \geq 0$ such that $\Gamma X_{\btheta_i} = \lambda' \Gamma X_{\btheta_j}$, which is impossible. Hence, $|E^{\Gamma,\btheta}|\leq 2$. Since $\btheta \in E^{\Gamma,\btheta}$, we get (i), and (ii) follows again from the fact that \eqref{eq:impossible si lambda > 0} has no solution if $\lambda\geq 0$.
We finally prove (iii). Let $\Gamma \in \cM^*_{\ell,K}(\Z_+)$ and $\btheta, \btheta' \in \cM_{crit}$ such that $\btheta' \in E^{\Gamma,\btheta}$. Since $\Gamma$ is not the null matrix and has nonnegative coefficients and $X_\btheta, X_{\btheta'}$ have only positive coefficients (by the Perron-Frobenius theorem) then both $\Gamma X_\btheta$ and $\Gamma X_{\btheta'}$ have nonnegative coefficients and are not $\bzero$. This means that necessarily $\lambda > 0$, which is not possible if $\btheta\not=\btheta'$ by \eqref{eq:impossible si lambda > 0}. This concludes the proof.
\end{proof}

We can now prove Proposition \ref{prop:opendirection2}.

\begin{proof}[Proof of Proposition \ref{prop:opendirection2}]
Consider $\overline{\btheta}$ such that $\mathcal{D}_{asy}^{\Gamma,\overline{\btheta}} \neq \{ \bzero \}$. Let $P$ be a linear map from $\R^K$ to $\R^{K-\ell}$ with kernel $\ker{P}=(\ker{\Gamma})^{\perp}=\text{Im}(\Gamma^{\intercal})$. We define the function $G^{\Gamma}:(0,+\infty)\times\mathcal{M}_{\text{crit}}\rightarrow \R^{K}$ by:
    \[
    G^{\Gamma}(\lambda,\btheta)= (\lambda \Gamma X_{\btheta}, P \chi(\btheta)).
    \]
Let us prove that this function is continuous. Since $\chi$ is continuous, $\btheta\rightarrow P \chi(\btheta)$ is continuous. To see that $\btheta\rightarrow \Gamma  X_{\btheta}$ is continuous, observe that, for any $\btheta \in \cM_{crit}$, we have the expression
\begin{align*}
    \text{com}(M_{\btheta}-I_K) = \frac{D_\btheta}{Y_\ell^\intercal Y_r} Y_\ell Y_r^\intercal,
\end{align*}
where $D_\btheta$ is any nonzero $(K-1)$-minor of $M_\btheta-I_K$ and $Y_\ell,Y_r$ denote respectively the $1$-left eigenvector and the $1$-right eigenvector of $M_\btheta$ with positive coordinates satisfying $||Y_\ell||_1=||Y_r||_1=1$. Such vectors exist by the Perron-Frobenius theorem. This implies that $X':=\frac{\text{com}(M_{\btheta}-I_K)\bone}{||\text{com}(M_{\btheta}-I_K)\bone||}$ is a $1$-left eigenvector of $M_\btheta$ satisfying $||X'||_1=1$, whose coordinates are either all positive or all negative. Using the notation $\vert Y\vert:= (\vert Y_1 \vert,\dots,\vert Y_n\vert)^{\intercal}$, the asymptotic direction $X_\btheta$ can be written as 
\begin{equation}
\label{eq:dorection asymp comatrice2}
X_{\btheta}=\frac{\vert\text{com}(M_{\btheta}-I_K)\bone\vert}{||\text{com}(M_{\btheta}-I_K)\bone||_1}.
\end{equation}
Now, since $\btheta \mapsto D_{\btheta}$ is continuous and always nonzero, it is either always positive or always negative. It is clear from \eqref{eq:dorection asymp comatrice2} that $\btheta \mapsto X_\btheta$ is a continuous function of $\btheta$. Hence, $\btheta \mapsto \Gamma X_\btheta$ (and thus $G^{\Gamma}$) is also continuous.

Recall now that, by Proposition \ref{prop:open} (ii), the set $\mathcal{M}_{\text{crit}}$ is a $(K-1)$-dimensional differentiable manifold. Hence, for any $\btheta\in \mathcal{M}_{\text{crit}}$ there exists a diffeomorphism $\psi_{\btheta}$ from $(-1,1)^K$ to an open subset $U_{\btheta}\subset \R^K$ such that $U_\btheta\cap \cM_{crit}=\psi_\btheta((-1,1)^{K-1}\times\{0\})$ and $\psi_\btheta(\bzero)=\btheta$.
 We can then define, for any $\varepsilon\in (0,1]$, the function $\tilde{G}^{\Gamma,\btheta,\varepsilon}:(0,+\infty)\times(-\varepsilon,\varepsilon)^{K-1}\rightarrow \R^{K}$ as:
\[
\tilde{G}^{\Gamma,\btheta,\varepsilon}(\lambda,x)= (\lambda \Gamma X_{\psi_{\btheta}(x)}, P \chi(\psi_{\btheta}(x))).
\]
Since $\mathcal{D}_{asy}^{\Gamma,\overline{\btheta}} \neq \{ \bzero \}$, by Lemma \ref{lem:0inDasy} we have $\Gamma X_{\btheta} \not=0$. Thus, one can find $\varepsilon>0$ small enough so that, for all $x\in (-\varepsilon,\varepsilon)^{K-1}\times\{0\},\ \Gamma X_{\psi_{\btheta}(x)}\not=0$. By Lemma \ref{lem:injective2} (i), for $\varepsilon$ small enough, $\tilde{G}^{\Gamma,\btheta,\varepsilon}$ is injective. It is also continuous and therefore by the invariance of domain theorem, its image is open.

Finally, notice that $\mathcal{D}_{asy}=\mathcal{D}^{I_K,\bzero}_{asy}=\mathcal{D}^{I_K,\btheta}_{asy}$ for all $\btheta \in \R^K$. Since $X_{\bzero} \neq \bzero$, Proposition \ref{prop:opendirection2} implies that $\mathcal{D}_{asy}$ is open. 
\end{proof}

\subsection{The function $\chi$ on $\mathcal{M}_{\text{crit}}$}

In this section, we show that the set $\mathcal{C}_{image}$ is closed and convex and that $\mathcal{M}_{\text{crit}}$ is sent by $\chi$ to the boundary of $\mathcal{C}_{image}$.

\begin{Proposition}\label{prop:convex}
Assume that $\bmu$ is entire, finite, nondegenerate, nonlocalized and irreducible. Then the set $\mathcal{C}_{image}$ is closed and convex.
\end{Proposition}

\begin{Proposition}\label{prop:bijection}
Assume that $\bmu$ is entire, finite, nondegenerate, nonlocalized and irreducible. Let $\partial \mathcal{C}_{image}$ be the boundary of $\mathcal{C}_{image}$. Then $\chi$ realizes a bijection between $\mathcal{M}_{\text{crit}}$ and the boundary $\partial \mathcal{C}_{image}$.    
\end{Proposition}

We start with a technical lemma that adds some quantitative information to Lemma \ref{lem:resultonc1}.

\begin{Lemma}\label{lem:directionstrongbound}
 Assume that $\bmu$ is entire, finite, nondegenerate, nonlocalized and irreducible. Let $\overline{\btheta}\in\cM_{crit}$. Then, there exists $a\in\R$ and $b>0$ such that:
 \[
\forall \, \btheta\in\R^K,\   X_{\overline{\btheta}}^{\intercal}\chi(\btheta)\geq a+b||\chi(\btheta)||_2.
\]
\end{Lemma}
\begin{proof}
Denote by $(e_1, \ldots, e_K)$ the canonical basis of $\R^K$. By Proposition \ref{prop:opendirection2}, there exists $\varepsilon>0$ such that, for all $i\in[K]$, there exists $(\lambda_i^{\pm},\overline{\btheta}_i^{\pm})\in(0,+\infty)\times \cM_{crit}$ such that $\lambda_i^{\pm}X_{\overline{\btheta}_i^{\pm}}=X_{\overline{\btheta}}\pm \sqrt{K}\varepsilon e_i$.
In particular the Euclidean ball $B_{\varepsilon}(X_{\overline{\btheta}})$ is included in the convex hull of $\{\lambda_i^{\pm} X_{\overline{\btheta}_i^{\pm}}\}$, and hence any $Y \in B_{\varepsilon}(X_{\overline{\btheta}})$ is a linear combination of them with nonnegative coefficients. Combining it with Lemma \ref{lem:resultonc1} (ii), we get, letting $A:=\min\{\lambda_i^{\pm}X^{\intercal}_{\overline{\btheta}_i^{\pm}}\chi(\overline{\btheta}_i^{\pm})\}$:
\[
\forall Y\in B_{\varepsilon}(X_{\overline{\btheta}}), \ \forall \btheta\in\R^K, \ Y^{\intercal}\chi(\btheta)\geq A.
\]

By taking $Y=X_{\overline{\btheta}}-\varepsilon \frac{\chi(\btheta)}{||\chi(\btheta)||_2}$ we get:
\[
\forall \btheta\in\R^K, \ X_{\overline{\btheta}}^{\intercal}\chi(\btheta)\geq A+\varepsilon||\chi(\btheta)||_2.
\]
\end{proof}

We now show that, if $\chi(\btheta)$ is on the boundary of $\mathcal{C}_{\text{image}}$, then $\bmu_\btheta$ is critical.

\begin{Lemma}\label{lem:uniquegammadir}
Assume that $\bmu$ is entire, finite, nondegenerate, nonlocalized and irreducible. Fix $\Gamma\in\cM_{\ell,K}^*(\Z)$ and $\overline{\btheta}\in\R^K$ such that $\bzero\not\in\mathcal{D}_{asy}^{\Gamma,\overline{\btheta}}$. Fix $X\in(0,+\infty)^K$ such that $\Gamma X \not=\bzero$. Let $\btheta'\in\R^K$ and assume that $\btheta'$ maximizes $f_X: \btheta\mapsto -X^{\intercal} \chi(\btheta)$ on the set $\{\btheta\sim_{\Gamma}\overline{\btheta}\}$. Then, $\bmu_{\btheta'}$ is critical and there exists $\lambda>0$ such that $\Gamma X=\lambda\Gamma X_{\btheta'}$.\\

In the specific case where $\Gamma=I_K$, $\bmu_{\btheta'}$ is critical and its asymptotic direction is $X$ (up to a multiplicative constant).
\end{Lemma}

\begin{proof}
Assume that $\btheta'$ is subcritical but not critical. The Jacobian of $\chi$ at $\btheta'$ is $M_{\btheta'}-I_K$, which is invertible. Hence, $\chi$ is locally invertible around $\btheta'$. Now, let $Y$ be such that $Y^{\intercal}\Gamma X >0$. There exists $\varepsilon>0$ and $\btheta'' \in \R^K$ such that $\chi(\btheta'')=\chi(\btheta')-\varepsilon \Gamma^{\intercal}Y$. This means that $\btheta''\sim_{\Gamma}\btheta'$ and $f_X(\btheta'')=f_X(\btheta')-\varepsilon X^{\intercal}\Gamma^{\intercal}Y<f_X(\btheta')$.  Hence, $\btheta'$ cannot be a maximizer of $f_X$ on the set $\{\btheta\sim_{\Gamma}\overline{\btheta}\}$.

Now, assume that $\btheta'$ is supercritical but not critical. By Lemma \ref{lem:fromsupercrittosubcrit}, there exists a subcritical (but not critical) $\btheta''$ such that $\chi(\btheta'')=\chi(\btheta')$ and in particular $\btheta'\sim_{\Gamma}\btheta''$. By the previous argument, $\btheta''$  cannot be a maximizer of $f_X$, so $\btheta'$ cannot either.
In the end, necessarily $\btheta'$ is critical.

Now let $X_{\btheta'}$ be the asymptotic direction of $\btheta'$. By Proposition \ref{prop:opendirection2}, for any $\varepsilon>0$ small enough, there exists $\lambda>0$ and $\overline{\btheta}^{\prime}\sim_{\Gamma}\btheta'$ such that $\bmu_{\overline{\btheta}^{\prime}}$ is critical and $\lambda \Gamma X_{\overline{\btheta}^{\prime}}=\Gamma X_{\btheta'}+\varepsilon \Gamma X$. 
Assume that $\Gamma X_{\btheta'}$ and $\Gamma X$ are not collinear. Then $\btheta' \neq \overline{\btheta}'$.
Observe now that, by Lemma \ref{lem:resultonc1} (ii):
\begin{equation}\label{eq:ineg1et2}
   X_{\overline{\btheta}^{\prime}}^{\intercal}\chi(\overline{\btheta}^{\prime})< X_{\overline{\btheta}^{\prime}}^{\intercal}\chi(\btheta') \text{ and }
      X_{\btheta'}^{\intercal}\chi(\btheta')< X_{\btheta'}^{\intercal}\chi(\overline{\btheta}^{\prime}).
\end{equation}

\noindent By definition of $\Gamma$-equivalence, there exists a vector $R\in\R^{\ell}$ such that $\chi(\btheta')-\chi(\overline{\btheta}')=\Gamma^{\intercal}R$. This means that we have:
\begin{equation}\label{eq:usualGammatrick}
\lambda X_{\overline{\btheta}^{\prime}}^{\intercal}(\chi(\overline{\btheta}^{\prime})-\chi(\btheta'))=
  \lambda X_{\overline{\btheta}^{\prime}}^{\intercal}\Gamma^{\intercal}R=
   \big(X_{\btheta'}+\varepsilon X\big)^{\intercal}\Gamma^{\intercal}R=
   \big(X_{\btheta'}+\varepsilon X\big)^{\intercal}(\chi(\overline{\btheta}^{\prime})-\chi(\btheta')).
\end{equation}
Furthermore, by assumption
\begin{equation}\label{eq:ineg3}
X^{\intercal}\chi(\btheta')\leq X^{\intercal}\chi(\overline{\btheta}^{\prime}),
\end{equation}
so by using the second inequality of \ref{eq:ineg1et2} and the inequality \ref{eq:ineg3} in Equation \ref{eq:usualGammatrick} we get:
\[
  \lambda X_{\overline{\btheta}^{\prime}}^{\intercal}(\chi(\overline{\btheta}^{\prime})-\chi(\btheta'))   >0.
\]
which is incompatible with the first inequality of Equation \ref{eq:ineg1et2}. We thus get the desired result. \\

Assume now that $\Gamma=I_K$. Then, necessarily $\bzero\notin \mathcal{D}_{asy}^{I_K,\overline{\btheta}}=\mathcal{D}_{asy}$ by definition of the asymptotic direction. We therefore get the desired result.
\end{proof}

The next lemma provides an affine bound to the function $\chi$. It will be used to show that, for $\btheta, \btheta' \in \R^K$, if $\chi(\btheta)$ and $\chi(\btheta')$ are close to each other then $\btheta$ and $\btheta'$ cannot be too far apart.

\begin{Lemma}
\label{lem:explicitboundonc}
Assume that $\bmu$ is entire, nonlocalized and irreducible. Let $\overline{\btheta}$ be such that $\bmu_{\overline{\btheta}}$ is critical. Then, there exists $a\in\R$ and $b\in(0,+\infty)$ such that:
\[
\forall \btheta\in\R^K,\ \ X^{\intercal}_{\overline{\btheta}} \, \chi(\btheta) \geq a+b ||\btheta-\overline{\btheta}||_2.
\]
\end{Lemma}
\begin{proof}
Let $$b:=\inf_{\btheta\in S^{K-1}} \left(f_{X_{\overline{\btheta}}}(\overline{\btheta})-f_{X_{\overline{\btheta}}}(\overline{\btheta}+\btheta)\right),$$ 
where $S^{K-1}$ is the unit sphere for $||\cdot||_2$ in $\R^K$.
As $\overline{\btheta}$ is the unique global maximizer of $f_{X_{\overline{\btheta}}}$, $f_{X_{\overline{\btheta}}}$ is continuous and $S^{K-1}$ is compact, we get that $b>0$.
For any $\btheta$ such that $||\btheta||_2\leq 1$, we have:
\[
f_{X_{\overline{\btheta}}}(\overline{\btheta}+\btheta)\leq f_{X_{\overline{\btheta}}}(\overline{\btheta}) \leq f_{X_{\overline{\btheta}}}(\overline{\btheta}) - b (||\btheta||_2-1).
\]
On the other hand, by Lemma \ref{lem:resultonc1} (i), the function $f_{X_{\overline{\btheta}}}: \btheta \mapsto -X^{\intercal}_{\overline{\btheta}} \chi(\btheta)$ is strictly concave. Thus, for any $\btheta$ such that $||\btheta||_2 > 1$, we get:
\begin{align*}
  ||\btheta||_2^{-1}  f_{X_{\overline{\btheta}}}(\overline{\btheta}+\btheta) + (1-||\btheta||_2^{-1}) f_{X_{\overline{\btheta}}}(\overline{\btheta}) 
  \leq f_{X_{\overline{\btheta}}}\left(\overline{\btheta}+||\btheta||_2^{-1} \btheta\right) 
  \leq f_{X_{\overline{\btheta}}}(\overline{\btheta})-b,
\end{align*}
which implies
\begin{align*}
f_{X_{\overline{\btheta}}}(\overline{\btheta}+\btheta)\leq f_{X_{\overline{\btheta}}}(\overline{\btheta}) - b ||\btheta||_2 \leq f_{X_{\overline{\btheta}}}(\overline{\btheta}) - b (||\btheta||_2-1).
\end{align*}

Setting $a:=-f_{X_{\overline{\btheta}}}(\overline{\btheta})-b$, we get the result.
\end{proof}

We now turn to the proof of Proposition \ref{prop:convex}.

\begin{proof}[Proof of Proposition \ref{prop:convex}]
We first show that the set $\mathcal{C}_{image}$ is closed. Let $(\btheta_n)_{n\in\N}$ be a sequence such that $\chi(\btheta_n)$ converges to some $\overline{\chi}$. Fix also $\btheta_c\in\mathcal{M}_{\text{crit}}$. By Lemma \ref{lem:explicitboundonc}, there exist constants $a,b$ with $b>0$, depending only on $\btheta_c$, such that:
\begin{equation}\label{eq:uselessstuff1}
X_{\btheta_c}^{\intercal}\chi(\btheta_n)\geq a + b ||\btheta_n-\btheta_c||_2.
\end{equation}
Since the sequence $\chi(\btheta_n)$ converges, it is bounded. Hence, \eqref{eq:uselessstuff1} implies that the sequence $(\btheta_n)_{n\in\N}$ is also bounded. We can therefore extract a subsequence that converges to some limit $\overline{\btheta}$. By continuity of the function $\chi$, we have $\chi(\overline{\btheta})=\overline{\chi}$ and therefore the set $\mathcal{C}_{image}$ is closed.\\

We now prove that the set $\mathcal{C}_{image}$ is convex. Define the set $\mathcal{C}^-$ as the image by $\chi$ of the set of subcritical (but not critical) parameters $\btheta\in\R^K$:
\[
\mathcal{C}^-:=\{\chi(\btheta),\ \rho_{\btheta}<1\}.
\]
We first show that $\mathcal{C}^-$ is convex. To this end, let $\btheta_1,\btheta_2\in \R^K$ such that $\rho_{\btheta_1},\rho_{\btheta_2}<1$. We define the set $E$ as follows:
\[
E:=\{s\in [0,1] \ | \ \exists \btheta\in\R^k \text{ with } \rho_{\btheta}<1 \text{ such that } \chi(\btheta)= s\chi(\btheta_1)+(1-s)\chi(\btheta_2)\}.
\]
Our goal is to show that $E=[0,1]$. By definition, $\{0,1\}\subset E$. We will now show that $E$ is both open and closed in $[0,1]$, which will imply the desired result. First, $E$ is open in $[0,1]$ by the implicit function theorem, since the Jacobian of $\chi$ at any $\btheta \in \R^K$ is $M_{\btheta}-I_K$, which is invertible if $\rho_{\btheta}<1$. To show that $E$ is closed, we use Lemma \ref{lem:explicitboundonc}. Set $\btheta_c\in\mathcal{M}_{\text{crit}}$. We define $m_1,m_2$ as follows:
\[
\forall i\in\{1,2\},\ m_i:=X_{\btheta_c}^{\intercal}\chi(\btheta_i).
\]
In particular, for any $s,\btheta$ such that $\chi(\btheta)=s\chi(\btheta_1)+(1-s)\chi(\btheta_2)$ we have:
\[
X_{\btheta_c}^{\intercal}\chi(\btheta)=sm_1+(1-s)m_2.
\]
By Lemma \ref{lem:explicitboundonc}, there exist constants $a,b$ with $b>0$ depending only on $\btheta_c$ such that:
\[
sm_1+(1-s)m_2=X_{\btheta_c}^{\intercal}\chi(\btheta)\geq a + b ||\btheta-\btheta_c||_2.
\]
This implies that:
\begin{equation}
\label{eq:borne sup}
||\btheta-\btheta_c||_2 \leq \frac{\max(m_1,m_2) - a}{b}.
\end{equation}

Now, fix $s\in (0,1)$ and assume that there exist two sequences $(\btheta_n)_{n\in\N}$, $(s_n)_{n\in\N}$ with $\btheta_n$ subcritical for all $n\in\N$, such that \[
\forall n\in\N,\ s_n\in[0,1] \text{ and } \chi(\btheta_n)=s_n\chi(\btheta_1)+(1-s_n)\chi(\btheta_2),
\]
and such that $s_n\rightarrow s$ as $n \rightarrow \infty$. By \eqref{eq:borne sup}, the sequence $(\btheta_n)_{n\in\N}$ is bounded and we can extract a subsequence that converges to some $\btheta_{lim}$, which is subcitical by continuity of $\cM_{\btheta}$ and its eigenvalues. By continuity of $\chi$, we have $\chi(\btheta_{lim})=s\chi(\btheta_1)+(1-s)\chi(\btheta_2)$. Observe that $\bmu_{\btheta_s}$ cannot be critical. Indeed $\btheta_s$ cannot be the unique maximizer of $f_{X_{\btheta_s}}$ as $f_{X_{\btheta_s}}(\btheta_s)=sf_{X_{\btheta_s}}(\btheta_1)+(1-s)f_{X_{\btheta_s}}(\btheta_2)$, which contradicts Lemma \ref{lem:resultonc1} (ii). From all this we get that $\mathcal{C}^-$ is convex.

We now define the two sets  $\mathcal{C}^+$ and $\mathcal{C}^1$ as follows:
\[
\begin{aligned}
\mathcal{C}^+:=&\{\chi(\btheta),\ \rho_{\btheta}>1\}\\
\mathcal{C}^1:=&\{\chi(\btheta),\ \rho_{\btheta}=1\}.
\end{aligned}
\]
We clearly have $\mathcal{C}_{image}=\mathcal{C}^-\cup\mathcal{C}^1\cup\mathcal{C}^+$. Furthermore, writing $\overline{\mathcal{C}^-}$ for the closure of $\mathcal{C}^-$, we claim that $\mathcal{C}^1\subset \overline{\mathcal{C}^-}$. Indeed by Proposition \ref{prop:open}, $\mathcal{M}_{\text{crit}}$ is a differentiable manifold of dimension $K-1$. Hence, for any $\btheta\in\mathcal{M}_{\text{crit}}$ there exists a sequence $(\btheta_n)_{n \in \N}$ converging to $\btheta$, with $\btheta_n\not\in\mathcal{M}_{\text{crit}}$ for all $n$. Therefore, by continuity of $\chi$, $\chi(\btheta_n)$ converges to $\chi(\btheta)$ and thus $\mathcal{C}^1\subset\overline{\mathcal{C}^-\cup\mathcal{C}^+}$. Since $\mathcal{C}^+\subset \mathcal{C}^-$ by Lemma \ref{lem:fromsupercrittosubcrit}, we have $\mathcal{C}^1\subset\overline{\mathcal{C}^-}$ and thus $\mathcal{C}_{image}\subset \overline{\mathcal{C}^-}$. Since $\mathcal{C}_{image}$ is closed, we have the equality $\mathcal{C}_{image}= \overline{\mathcal{C}^-}$. Finally, the closure of a convex set is convex, so $\mathcal{C}_{image}$ is convex. 
\end{proof}

We finally prove Proposition \ref{prop:bijection}.

\begin{proof}[Proof of Proposition \ref{prop:bijection}]
First, by Proposition \ref{prop:convex}, the set $\mathcal{C}_{image}$ is closed. By Lemma \ref{lem:resultonc1} (ii), if $\btheta\in\mathcal{M}_{\text{crit}}$, then $\chi(\btheta)$ is not in the interior of $\mathcal{C}_{image}$ and therefore it belongs to $\partial \mathcal{C}_{image}$. Conversely by Lemma \ref{lem:fromsupercrittosubcrit}, if $\btheta\not\in\mathcal{M}_{\text{crit}}$ there exists a subcritical but not critical $\btheta'$ such that $\chi(\btheta')=\chi(\btheta)$. Then, since the Jacobian of $\chi$ at $\btheta'$ is $M_{\btheta'}-I_K$ and $\rho_{\btheta'}<1$, $\chi$ is locally invertible at $\btheta'$ and therefore $\chi(\btheta')$ is in the interior of $\mathcal{C}_{image}$. From this we get that $\chi(\mathcal{M}_{\text{crit}})=\partial \mathcal{C}_{image}$ and $\mathcal{M}_{\text{crit}}=\chi^{-1}(\partial \mathcal{C}_{image})$. Finally, by Lemma \ref{lem:resultonc1} (ii), $\chi$ restricted to $\mathcal{M}_{\text{crit}}$ is injective.
\end{proof}

\section{Proof of the main result}
\label{sec:proof}

It turns out that, when $rk(\Gamma) \geq 2$, by Proposition \ref{prop:open} there may not be uniqueness of the critical $\Gamma$-equivalent exponential tilting of a given distribution $\bmu$. However, we have proved (see Corollary \ref{cor:asymp dir are not collin}) that these critical tiltings are characterized by their asymptotic directions, in the sense that two such tiltings necessarily have different asymptotic directions.

The goal of this section is to characterize which directions are possible asymptotic directions of a critical $\Gamma$-equivalent tilting of $\bmu$, and for which parameters $\btheta$. This will help us prove Theorem \ref{thm:accessdir2}.

Let us first define the notion of accessible (asymptotic) direction.

\begin{Definition}\label{def:accessible}
Fix a projection $\bmu$, and let $I\subseteq [K]$ be the set of types that have no child under $\bmu$ with positive probability (that is, $I=\{ i \in [K], \mu^{(i)}(\bzero)>0 \}$). Let $X\in[0,+\infty)^K\backslash\{\bzero\}$. We say that the direction $X$ is:
\begin{itemize}
\item \textbf{accessible} if there exists a tree $T$ such that
\begin{itemize}
    \item $\Prob_{\bmu}(\cT=T)>0$;
    \item $T$ has at least $2$ vertices;
    \item the root of $T$ has type $i\in I$;
    \item at least one of its leaves has the same type $i$ as the root;
    \item the vectors $\tilde{\bN}(T) := (\tilde{N}_1(T),\dots,\tilde{N}_K(T))^{\intercal}$ and $X$ are collinear, where $\tilde{N}_j(T)$ is the number of nodes of type $j$ in $T$ (counting the leaves but excluding the root).
\end{itemize}
\item \textbf{strongly accessible} if it is in the interior of the convex hull of the accessible directions.
\end{itemize}
We define $\mathcal{D}_{sad}$ as the set of strongly accessible directions (up to a positive multiplicative constant).
\end{Definition}

The set of strongly accessible directions turns out to be exactly the set of asymptotic directions (up to a multiplicative constant).

\begin{Theorem}
\label{thm:accessdir}
Assume that $\bmu$ is entire, finite, nondegenerate, nonlocalized and irreducible. Then, the set $\mathcal{D}_{asy}$ is both open and convex and $\mathcal{D}_{sad}=\mathcal{D}_{asy} $. 
\end{Theorem}

After stating some preliminary results in Section \ref{ssec:asymptotic proportion of types}, we prove Theorem \ref{thm:accessdir} in Section \ref{ssec:proofpresquemain}, deduce from it Theorem \ref{thm:accessdir2}, and finally prove Corollary \ref{cor:coefsinz}.

\subsection{Asymptotic proportion of types}
\label{ssec:asymptotic proportion of types}

We start by showing that, in a $\bze_\btheta$-BGW tree conditioned on its number of root-type vertices, the number of vertices of each type is close to (a multiple of) the asymptotic direction $X_\btheta$.

\begin{Proposition}\label{prop:behaviourisasdir}
    Let $\bmu$ be entire, finite and irreducible. Let $\btheta \in \R^K$ such that $\bmu_\btheta$ is critical. Then, for any $\varepsilon>0$, any $N_0>0$, there exists $N>N_0$ such that $\Prob\left( N_1(\cT_n^{\btheta,(1)})=N\right)>0$ and, under the conditioning $N_1(\mathcal{T}^{\btheta,(1)})=N$, one has:
    \[
    \Prob_{\bmu_{\btheta}}\left(\forall i\in \{2,\dots,K\},\ \left\vert X_{\btheta}(1) N_i(\mathcal{T}^{\btheta,(1)})- X_{\btheta}(i) N_1(\mathcal{T}^{\btheta,(1)})\right\vert \leq \varepsilon N \, \, \big\vert \, N_1(\mathcal{T}^{\btheta,(1)})=N \right) \geq \frac{1}{2}.
    \]
Here, we recall that $\cT^{\btheta,(1)}$ denotes the tree $\cT^{\btheta}$ conditioned to have a root of type $1$.
\end{Proposition}

Let us immediately mention a useful corollary of this result.

\begin{Corollary}
    \label{cor:accessible directions are close to asymptotics}
Let $\btheta \in \cM_{crit}$, and fix $\varepsilon>0$. Then, there exists an accessible direction $X$ such that $||X-X_\btheta||_1 \leq \varepsilon$.  
\end{Corollary}

\begin{proof}
Fix $\btheta\in\mathcal{M}_{\text{crit}}$. By Proposition \ref{prop:behaviourisasdir}, for any $\varepsilon>0$, there exists $i_0 \in I$ (the set of types $i \in [K]$ such that a vertex of type $i$ has zero child with positive probability) and a sequence $(J_n)_{n \geq 1}$ going to $+\infty$ such that, for all $n$, $\Prob(N_1(\cT_n^{(1)})=J_n)>0$ and
\begin{align}
\label{eq:Jn grand}
\Prob_{\bmu_\btheta}^{J_n} \bigg( \cT^{(1)} \text{ has a leaf of type } i_0 \text{ and } \forall i\in \{2,\dots,K\},\ \frac{\big\vert X_{\btheta}(1) N_i(\mathcal{T}^{(1)})- X_{\btheta}(i) N_1(\mathcal{T}^{(1)})\big\vert}{J_n} \leq \varepsilon \bigg) \geq \frac{1}{2K},
\end{align}
where $\Prob_\mu^J$ denotes the probability measure under $\bmu_\btheta$ conditioned on $\{N_1(\mathcal{T}^{(1)})=J\}$.

Now, by irreducibility of $\bmu$, there exists a finite tree $T_0$ with root of type $i_0$ and at least one vertex of type $1$, such that $\Prob(\cT^{(i_0)}=T_{0})>0$. For any tree $T$, we can look at the concatenation $T_0\cdot T$ (the tree obtained by replacing one of the leaves of type $i$ in $T_0$ by $T$) with root $i_0$. By \eqref{eq:Jn grand} for $J_n$ large enough, using the fact that $\tilde{\bN}(T_0\cdot T)=\tilde{\bN}(T_0)+\tilde{\bN}(T)$, we get that, for any $\varepsilon>0$, there exists $X$ accessible such that $||X-X_\btheta||_1<\varepsilon$.
\end{proof}

In order to prove Proposition \ref{prop:behaviourisasdir}, for all $j' \in [K]$, consider the nonconditioned tree $\cT^{\btheta,(j')}$ with root of type $j'$, and denote by $t^{(j')}$ its shape. For all $j \in \{2,\ldots,K\}$, we define
\begin{align*}
    N_j^{(j')} := \left|\left\{u \in \cT^{\btheta,(j')}, \be_{t^{(j')}}(u)=j, \forall v \in \rrbracket \varnothing,u \rrbracket, \be_t(v) \neq 1\right\} \right|, 
\end{align*}
the number of vertices of $\cT^{\btheta,(j')}$ with type $\neq 1$ having no ancestor of type $1$ except possibly the root of the tree.

\begin{Lemma}
\label{lem:expectationinblob}
Let $\bmu$ be entire, finite and irreducible. For all $2 \leq j \leq K$, we have
    \begin{align*}
        \E[N_j^{(1)}]=\frac{X_\btheta(j)}{X_\btheta(1)}.
    \end{align*}
\end{Lemma}

Although a proof can be found in \cite[Proof of Proposition 4]{Mie08}, we decided to give one here for completeness.

\begin{proof}
Observe that we have the following relations:
\begin{align*}
    \E[N_j^{(j')}] &= \mathds{1}_{j'=j} + \sum_{k_1, \ldots, k_K \geq 0} \mu_\btheta^{(j')}(k_1, \ldots, k_K) \sum_{i=2}^K \E[N_j^{(i)}]\\
    &= \mathds{1}_{j'=j} + \sum_{k_1, \ldots, k_K \geq 0} \mu_\btheta^{(j')}(k_1, \ldots, k_K) \sum_{i=1}^K k_i \E[N_j^{(i)}] - \sum_{k_1, \ldots, k_K \geq 0} \mu_\btheta^{(j')}(k_1, \ldots, k_K) k_1 \E[N_j^{(1)}]\\
    &= \mathds{1}_{j'=j} + \sum_{i=1}^K \E[N_j^{(i)}] M_\btheta(j',i) - \E[N_j^{(1)}] M_\btheta(j',1). 
\end{align*}

Now define the matrix $P:=\left( \E[N_j^{(j')}] \right)_{1 \leq j,j' \leq K}$, let $P^{(1)}$ be its first column, and let $M_\btheta^{\intercal,(1)}$ be the first row of $M_\btheta^\intercal$. We can write:
\begin{align*}
    P = I_K + P M_\btheta^\intercal - P^{(1)} M_\btheta^{\intercal,(1)}.
\end{align*}

Multiplying to the right by $X_\btheta$ and using that it is a $1$-left eigenvector of $M_\btheta$, we get
\begin{align*}
    PX_\btheta + P^{(1)} M_\btheta^{\intercal,(1)} X_\btheta &= X_\btheta + P M_\btheta^\intercal X_\btheta \\
    PX_\btheta + X_\btheta(1) P^{(1)} &= X_\btheta + PX_\btheta\\
    P^{(1)} = \frac{1}{X_\btheta^{(1)}} X_\btheta.
\end{align*}
The result follows.
\end{proof}

We can now prove Proposition \ref{prop:behaviourisasdir}.

\begin{proof}[Proof of Proposition \ref{prop:behaviourisasdir}]
In this proof, for convenience, we write $\Prob$ for $\Prob_{\bmu_{\btheta}}$. We know by e.g. ~\cite[Section $4.3.1$]{Ste18} that there exist a constant $c>0$ and an increasing sequence $(N_n)_{n\in\N}$ such that
\begin{align*}
    \Prob\left( N_1(\cT^{\btheta,(1)})=N_n \right) \underset{n \rightarrow \infty}{\sim} c N_n^{-3/2}.
\end{align*}
In the rest of the proof, we let $N$ go to $+\infty$ along such a subsequence. We use the ideas of \cite[Proposition 3.3 and Lemma 3.5]{ABST25+}. Consider a sequence of i.i.d. $K$-tuples $\left((\xi_j^{(p)})_{1 \leq j \leq K}\right)_{p \geq 1}$ of random variables distributed as $( N_j^{(1)})_{1 \leq j \leq K} $ and fix $i \in \{1, \ldots, K\}$. Using the cycle lemma (see e.g.~\cite[Lemma 15.3]{Jan12}), we have that, for all $\ell \geq 1$ such that $|\ell-N \frac{X_\btheta(i)}{X_\btheta(1)}|>\varepsilon N$:

\begin{align*}
    \Prob\left( N_1(\cT^{(1)})=N, N_i(\cT^{(1)})=\ell \right) &= \frac{1}{N} \Prob\left(\sum_{p=1}^{N} \xi_1^{(p)}=N-1, \sum_{p=1}^N \xi_i^{(p)} = \ell \right)\\
    &\leq \Prob\left( \sum_{p=1}^N \xi_i^{(p)} = \ell \right)\\
    &\leq \Prob\left( \sum_{p=1}^N \left( \xi_i^{(p)} - \E[\xi_i^{(p)}] \right)=\ell-N\frac{X_\btheta(i)}{X_\btheta(1)} \right) \text{(by Lemma \ref{lem:expectationinblob})}\\
    &\leq \Prob\left( \Big\vert\sum_{p=1}^N \left( \xi_i^{(p)} - \E[\xi_i^{(p)}] \right)\Big\vert\geq \Big\vert\ell-N\frac{X_\btheta(i)}{X_\btheta(1)}\Big\vert \right).
\end{align*}
The variables $(\xi_i^{(p)}-\E[\xi_i^{(p)}])$ are now centered and it can be checked (see \cite[Prop. 4(i),(iii)]{Mie08}) that $\xi_i^{(p)}$ has exponential moments of all order.  We conclude by a moderate deviation estimate (see e.g. \cite[Example 1.4]{RS15} or \cite[Lemma 3.5]{ABST25+}) that there exists $D>0$ such that, for any $\rho$ small enough, any $x>0$:
\begin{align*}
\Prob\left( \Big\vert\sum_{p=1}^N \left( \xi_i^{(p)} - \E[\xi_i^{(p)}] \right)\Big\vert\geq x \right) \leq \exp \left( D N \rho^2 - \rho x \right).
\end{align*}

Taking $\rho =\varepsilon (2D)^{-1}$, we get that

\begin{align}
\label{eq:grandes dev}
&\sum_{\big\vert\ell-N \frac{X_\btheta(i)}{X_\btheta(1)}\big\vert \geq \varepsilon N} \Prob\left( \Big\vert\sum_{p=1}^N \left( \xi_i^{(p)} - \E[\xi_i^{(p)}] \right)\Big\vert\geq \Big\vert\ell-N\frac{X_\btheta(i)}{X_\btheta(1)}\Big\vert \right) \nonumber \\
&\qquad \leq N \left( \frac{X_\btheta(i)}{X_\btheta(1)} -\varepsilon \right) \exp \left( - \frac{\varepsilon^2}{2D} N \right) + \sum_{a \geq 0} \exp\left( -\frac{\varepsilon^2}{2D} N - \frac{\varepsilon}{2D} a \right) \nonumber \\
&\qquad = O\left(\exp(-BN) \right),
\end{align}
for some $B>0$. The result follows.
\end{proof}

\subsection{Proof of Theorem \ref{thm:accessdir}}
\label{ssec:proofpresquemain}

Before proving Theorem \ref{thm:accessdir}, we need a couple of technical lemmas. We first show that we can bound uniformly from below the quantities $X^{\intercal}\chi(\btheta)$, for accessible directions $X$.

\begin{Lemma}\label{lem:accessiblebound}
Assume that $\bmu$ is entire and finite. Let $X\in[0,+\infty)^K$ be an accessible direction. Then, there exists a constant $m_X \in \R$ such that:
\[
\forall \btheta\in\R^K, \ X^{\intercal}\chi(\btheta) \geq m_X.
\]
\end{Lemma}
\begin{proof}
We use here the notation of Definition \ref{def:accessible}. Let $X$ be an accessible direction and fix $i\in I$. Let $T$ be a nontrivial tree such that $\tilde{\bN}(T)$ and $X$ are collinear, the root of $T$ is of type $i$, all the leaves of $T$ are of a type in $I$ and at least one is of type $i$. We arbitrarily choose one such leaf and call it $x$. Let $p>0$ be the probability of such a tree $T$ under $\bmu$ (that is, $p=\Prob(\cT^{(i)}=T)$), For all $n \geq 1$, let $T_n$ be the concatenation of $n$ such trees ($T_1=T$ and $T_{n+1}$ is obtained by grafting a copy of $T_n$ on $x$ in $T$), and define $p^{(n)} := \Prob(\cT^{(i)}=T_n)$. Let $p^0_i>0$ the probability that a node of type $i$ has no child. Then, we have $p^{(n)}=p^0_i \Big(\frac{p}{p^0_{i}}\Big)^n$. We also have by definition of $\chi$, setting $p^{(n)}_{\btheta}=\Prob_{\bmu_{\btheta}}(\cT^{\btheta,(i)}=T_n)$:
\[
p^{(n)}_{\btheta}=e^{-\theta_i}p^{(n)} e^{-n\tilde{\bN}(T)^{\intercal}\chi(\btheta)}=e^{-\theta_i}p^0_i\Big(\frac{p}{p^0_{i}}e^{-\tilde{\bN}(T)^{\intercal}\chi(\btheta)}\Big)^n.
\]
As this quantity must be smaller than $1$ for all $n \geq 1$, we have:
\[
\frac{p}{p^0_{i}}\leq e^{\tilde{\bN}(T)^{\intercal}\chi(\btheta)}
\]
This implies the desired result.
\end{proof}

We need a last result before proving Theorem \ref{thm:accessdir}.

\begin{Lemma}\label{lem:accessiblestrongbound}
Assume that $\bmu$ is entire and finite. Let $X\in(0,+\infty)^K\cap \mathcal{D}_{sad}$ be a strongly accessible direction. Then, there exists $a\in\R$ and $b,\varepsilon>0$ such that:
\[
\forall \btheta\in\R^K,\ \forall Y\in B_{\varepsilon}(X),\ Y^{\intercal}\chi(\btheta)\geq a+b||\chi(\btheta)||_2.
\]
\end{Lemma}

\begin{proof}
    If $X$ is strongly accessible, then there exist $\lambda\in(0,+\infty)$, $N \geq 1$ and $(X_i)_{1\leq i \leq N}$ such that, for all $i\in[N]$, $X_i$ is accessible and the convex hull of $\{X_i\}_{i \in [N]}$ contains a ball of positive radius centered on $\lambda X$.
    By Lemma \ref{lem:accessiblebound}, there exists a constant $m^-$ such that for all $i\in [N]$, we have:
    \[
    \forall \btheta\in\R^K, \ X_i^{\intercal}\chi(\btheta) \geq m^-.
    \]
    From this we get that there exists $M^-$ and $\varepsilon>0$ such that 
    \[
    \forall \btheta\in\R^K, \ \forall Y \in B_{\varepsilon}(X),\  \ Y^{\intercal}\chi(\btheta) \geq M^-.
    \]
    In particular, by taking $Y=X-\varepsilon \chi(\btheta)/||\chi(\btheta)||_2$ we get:
    \[
    X^{\intercal}\chi(\btheta) \geq M^-+\varepsilon||\chi(\btheta)||_2.
    \]
    So for all $Y\in B_{\varepsilon/2}(X)$ we have:
    \[
    Y^{\intercal}\chi(\btheta) \geq M^-+\varepsilon||\chi(\btheta)||_2 + (Y-X)^{\intercal}\chi(\btheta) \geq M^-+\frac{\varepsilon}{2}||\chi(\btheta)||_2.
    \]
\end{proof}

We now have all the tools to prove Theorem \ref{thm:accessdir}.

\begin{proof}[Proof of Theorem \ref{thm:accessdir}]

We divide the proof into two parts.

\textbf{1) The set $\mathcal{D}_{asy}$ is open and convex.} 

First, $\mathcal{D}_{asy}$ is open by Proposition \ref{prop:opendirection2}. Now fix $\btheta_1, \btheta_2\in\mathcal{M}_{\text{crit}}$. Since $\mathcal{D}_{asy}$ is a cone, we only need to prove that the set $$E := \{s\in[0,1], sX_{\btheta_1}+(1-s)X_{\btheta_2}\in \mathcal{D}_{asy}\}$$ 
is equal to the whole interval $[0,1]$. Observe first that $E$ is open in $[0,1]$ by Proposition \ref{prop:opendirection2} and is nonempty as it contains $0$ and $1$. Hence, we only have to show that $E$ is closed.

By Lemma \ref{lem:explicitboundonc} applied to $\btheta_1$ and $\btheta_2$, there exist $a'\in\R$ and $b'>0$ such that:
\begin{equation}
\label{eq:borne inf theta 1 theta 2}
\forall \btheta\in \R^K,\ \forall s\in[0,1],\ (sX_{\btheta_1}+(1-s)X_{\btheta_2})^{\intercal}\chi(\btheta)\geq a' + b'||\btheta||_2.
\end{equation}
If there exist $\btheta \in \cM_{crit}$ and $\lambda>0$ such that $X_{\btheta}=\lambda(sX_{\btheta_1}+(1-s)X_{\btheta_2})$, then by Lemma \ref{lem:resultonc1} (ii) and the fact that $\chi(\bzero)=\bzero$, we have $(sX_{\btheta_1}+(1-s)X_{\btheta_2})^{\intercal}\chi(\btheta)\leq 0$ and therefore $||\btheta||_2\leq -\frac{a'}{b'}$ by \eqref{eq:borne inf theta 1 theta 2}. 
Therefore, the set 
$$\{\btheta\in\cM_{crit}, \exists s \in [0,1],\lambda\in(0,+\infty),\ X_{\btheta}=\lambda(sX_{\btheta_1}+(1-s)X_{\btheta_2})\}$$ 
is bounded. 

Using the fact that the asymptotic direction is a continuous function of $\btheta$ (again since the asymptotic direction can be written as $X_{\btheta}=\text{com}(M_{\btheta}-I_n)\bone /||\text{com}(M_{\btheta}-I_n)\bone||_2$, see the proof of Proposition \ref{prop:opendirection2}), the set $E$ is closed and therefore it is $[0,1]$. 
As a consequence, $\mathcal{D}_{asy}$ is convex.\\

\textbf{2) $\mathcal{D}_{asy}=\mathcal{D}_{sad}$}

Let us show first that $\mathcal{D}_{asy}\subseteq\mathcal{D}_{sad}$. 
Fix $\btheta_0 \in \cM_{crit}$. We want to prove that $X_{\btheta_0} \in \mathcal{D}_{sad}$. Since $\mathcal{D}_{asy}$ is open by Proposition \ref{prop:opendirection2}, there exists $\alpha>0$ and $\btheta^{\pm}_1,\dots,\btheta^{\pm}_K \in \cM_{crit}$, $\lambda^{\pm}_1, \ldots, \lambda^{\pm}_K>0$ such that $\lambda^{\pm}_i X_{\btheta^{\pm}_i}:=X_{\btheta_0}\pm \alpha e_i$. By Corollary \ref{cor:accessible directions are close to asymptotics}, for any $\varepsilon>0$, there exist accessible directions $Y^{\pm}_1,\dots,Y^{\pm}_K$ such that $||Y^{\pm}_i-\lambda^{\pm}_i X_{\btheta^{\pm}_i}||_1\leq \varepsilon$. Finally, by taking $\varepsilon$ small enough, there exists a small ball centered on $X_{\btheta}$ in the convex hull of $\{Y^{\pm}_1,\dots,Y^{\pm}_K\}$ and therefore, by definition of $\mathcal{D}_{sad}$ we get that $X_{\btheta}\in\mathcal{D}_{sad}$. \\ 

Now we show that $\mathcal{D}_{sad}\subseteq\mathcal{D}_{asy}$. Fix $X\in\mathcal{D}_{sad}$. By Lemma \ref{lem:accessiblestrongbound}, there exist $a\in\R$ and $b>0$ such that:
\begin{equation}
\label{eq:chitheta}
\forall \btheta\in\R^K, \ X^{\intercal} \chi(\btheta)\geq a +b||\chi(\btheta)||_2.
\end{equation}
Fix $\btheta_0 \in \cM_{crit}$. By the Cauchy-Schwartz inequality, we have $||\chi(\btheta)||_2\geq X_{\btheta_0}^{\intercal}\chi(\btheta)/||X^\intercal_{\btheta_0}||_2$. Hence, combining \eqref{eq:chitheta} and Lemma \ref{lem:explicitboundonc} applied to $\btheta_0$, there exist $a'\in\R$ and $b'>0$ such that:
\[
\forall \btheta\in\R^K, \ X^{\intercal}\chi(\btheta)\geq a'+b'||\btheta||_2.
\]

\noindent Hence, there exists $\btheta_{X} \in \R^K$ that minimizes $\btheta\mapsto X^{\intercal}\chi(\btheta)$. This implies by Lemma \ref{lem:uniquegammadir} that $\mathcal{D}_{sad}\subseteq \mathcal{D}_{asy}$.
\end{proof}

We can now prove Theorem \ref{thm:accessdir2}.

\begin{proof}[Proof of Theorem \ref{thm:accessdir2}]
First, $\mathcal{D}_{asy}$ is nonempty by Proposition \ref{prop:Mcrit is nonempty}, and open and convex by Theorem \ref{thm:accessdir}.\\
Now, fix $\Gamma\in\cM^*_{\ell,K}(\Z)$. Without loss of generality, we can assume that $\Gamma$ is of rank $\ell$. Fix $\overline{\btheta}\in\R^K$ and recall the definition:
\[
\mathcal{D}^{\Gamma,\overline{\btheta}}_{asy}:=\{\lambda \Gamma X_{\btheta} |\lambda\in(0,+\infty),\btheta\in\mathcal{M}_{\text{crit}}, \btheta \sim_\Gamma \overline{\btheta}\}.
\]
If $\bzero\in\mathcal{D}^{\Gamma,\overline{\btheta}}_{asy}$ then by Proposition \ref{prop:opendirection2}, $\mathcal{D}_{asy}^{\Gamma,\overline{\btheta}}=\{\bzero\}$. In this case, by Lemma \ref{lem:0inDasy}, $\overline{\btheta}$ is critical, $\Gamma X_{\overline{\btheta}}=\bzero$ and $\left(\btheta\sim_{\Gamma} \overline{\btheta} \implies \btheta= \overline{\btheta}\right)$.\\
On the other hand, assume that $\bzero\not\in \mathcal{D}^{\Gamma,\overline{\btheta}}_{asy}$. Set $X\in \mathcal{D}_{asy}$ such that $\Gamma X\not=\bzero$. By Theorem \ref{thm:accessdir}, $X \in \mathcal{D}_{sad}$. Observe that the function $f_X$ is continuous and $\{\btheta \ | \ \btheta\sim_{\Gamma} \overline{\btheta}\}:=\chi^{-1}(\chi(\overline{\btheta})+Im(\Gamma^\intercal))$ is closed. Hence, by Lemma \ref{lem:explicitboundonc}, there exists $\btheta' \in \R^K$ maximizing $f_X$ on $\{\btheta, \btheta\sim_{\Gamma} \overline{\btheta}\}$. By Lemma \ref{lem:uniquegammadir}, $\bmu_{\btheta'}$ is necessarily critical and its asymptotic direction $X_{\btheta'}$ satisfies
\begin{equation}
\Gamma X_{\btheta'} = \lambda \Gamma X \text{ for some } \lambda >0.
\end{equation}
Finally, by Lemma \ref{lem:injective2} this maximizer is unique.
\end{proof}

We finally prove Corollary \ref{cor:coefsinz}.

\begin{proof}[Proof of Corollary \ref{cor:coefsinz}]
Fix $\Gamma\in\cM_{\ell,K}^*$. By Theorem \ref{thm:accessdir2} there are two possible cases:
\begin{itemize}
    \item either $\mathcal{D}^{\Gamma,\bzero}_{asy}=\{\bzero\}$; in this case, by Theorem \ref{thm:accessdir2}, $\bmu = \bmu_\bzero$ is critical and the result follows immediately;
    \item or $\bzero \notin \mathcal{D}^{\Gamma,\bzero}_{asy}$. In this case, by Theorem \ref{thm:accessdir2}, $\mathcal{D}_{asy}$ is open and nonempty, so that necessarily there exists $X \in \mathcal{D}_{sad}$ such that $\Gamma X \neq \bzero$. By Theorem \ref{thm:accessdir2} again, there exists $\btheta$ critical such that $\btheta\sim_{\Gamma} \bzero$.
\end{itemize}
\end{proof}

\section{Application to local limits of noncritical trees}
\label{sec:locscallim}

In this section, we prove Theorem \ref{thm:critical local limit} as a consequence of Theorem \ref{thm:accessdir2}. Let $\bze$ be an entire offspring distribution and $\Gamma \in \cM^*_{\ell,K}(\Z)$. The main idea is to show that, in the specific case where $\bze$ is critical and $(\bk(n))_{n \geq 1} \in (\Z_+^\ell)^\N$ is such that $\bk(n)/||\bk(n)||_1 \rightarrow X$ (where $X$ denotes the $1$-left eigenvector associated to $\bze$), then the conditioned trees $\cT^{(i)}_{\Gamma,\bk(n)}$ converge locally as $n \rightarrow\infty$. This is done in Section \ref{ssec:local limits of critical trees}, making use of \cite{ADG18}. If either $\bk(n)/||\bk(n)||_1$ does not converge to $X$ or $\bze$ is not critical, we use Theorem \ref{thm:accessdir2} to show that there exists a critical distribution that is $\Gamma$-equivalent to $\bze$, which is enough to conclude.

\subsection{Multitype Kesten trees}

We construct here the infinite discrete trees that appear as local limits of critical multitype BGW trees. It turns out that they all share a common structure: a unique end (infinite spine), on which are grafted independent multitype trees that are identically distributed conditionally on the type of their root. In regard of Kesten's seminal work \cite{Kes86}, we will call these trees multitype Kesten trees. This multitype construction was first introduced in \cite{KLPR97}, see also \cite[Proposition $3.1$]{Ste18} for a proof in the broader case of mutitype forests.

\begin{Definition}
Let $\bze$ be a $K$-type critical distribution, and recall that $\bb$ denotes the renormalized right $1$-eigenvector of the mean matrix $M$. Denote by $\hz$ the biased family of distributions defined as:
\begin{align*}
\forall j \in [K], \forall \bx \in \cW_K, \hz^{(j)}(\bx) = \frac{1}{b_j} \sum_{k=1}^{|\bx|} b_{x_k} \zeta^{(j)}(\bx),
\end{align*}
where $|\bx|$ denotes the length of $\bx$. In particular, $\hz^{(j)}(\varnothing) = 0$.
Given a type $i \in [K]$, we define the tree $\cT_*(\bze)$ as follows: it is made of a spine, which is an infinite branch starting from the root. On this infinite branch, vertices have distribution $\hz$. Given an element $v$ of the spine, denote by $\bw_v$ its ordered list of offspring types. Then, the probability that the child of $v$ belonging to the infinite spine is $vj$ (that is, the $j$-th of its children) is proportional to $b_{\ell(vj)}$ - that is, equal to
\begin{align*}
\frac{b_{\ell(vj)}}{\sum_{i=1}^{|\bw_v|} b_{\ell(vi)}},
\end{align*}
where we recall that $\ell(u)$ is the type of the vertex $u$.

Finally, for all $i \in [K]$, on any offspring of type $i$ of a vertex of the spine that is not itself on the spine, we graft a tree $\cT^{(i)}$ with root of type $i$, which is independent of all the rest of the tree.
\end{Definition}

In the monotype case, the child of a vertex on the spine that will be itself on the spine is just chosen uniformly at random. Observe also that, since $\hz^{(j)}(\varnothing)=0$ for all $j \in [K]$, the spine is indeed infinite.

\subsection{Local limits of critical conditioned trees}
\label{ssec:local limits of critical trees}

We prove here the local convergence of conditioned trees, under the additional aperiodicity condition (Definition \ref{def:aperiodic}). We start by proving it in the specific case where $\bze$ is critical and $\bk(n)/||\bk(n)||_1$ converges to the $1$-left eigenvector $X$ of $M$ with positive coordinates and such that $||X||_1=1$. We then use the results of the previous sections to extend it to noncritical distributions and other asymptotic directions.

A distribution $\bze$, a matrix $\Gamma \in \cM_{\ell,K}^*(\Z)$ and a vector $\bk \in \Z^\ell$ being given, we denote by $\cT^{(1)}_{\Gamma,\bk}$ the multitype tree with offspring distribution $\bze$ and root of type $1$, conditioned on $\Gamma  \bN(\cT) = \bk^\intercal$.

\begin{Theorem}
\label{thm:critical local limit cas facile}
    Let $\bze$ be a critical, aperiodic, finite, nondegenerate, nonlocalized and irreducible distribution, and $X$ the renormalized $1$-left eigenvector of $M$. Let $\Gamma \in \cM^*_{\ell,K}(\Z)$ of rank $\ell$ be such that $\Gamma X \neq \bzero$. Let $(\bk(n))_{n\geq 1}$ be a sequence of $K$-tuples of integers such that
    \begin{itemize}
        \item $||\bk(n)||_1 \underset{n \rightarrow \infty}{\rightarrow} +\infty$;
        \item As $n \rightarrow \infty$,
       $ \frac{\bk(n)}{||\bk(n)||_1} \underset{\rightarrow}{\rightarrow} X$;
        \item for all $n \geq 1, \Prob\left( \Gamma  N(\cT^{(1)})=\bk(n) \right) > 0$.
    \end{itemize}
Then, we have:
\begin{align*}
    \cT^{(1)}_{\Gamma,\bk(n)} \underset{n \rightarrow \infty}{\overset{(d),loc}{\rightarrow}} \cT_*^{(1)},
\end{align*}
where $\cT_*^{(1)}$ is the multitype Kesten tree associated to $\bze$.
\end{Theorem}

Let us immediately show how it implies Theorem \ref{thm:critical local limit}.

\begin{proof}[Proof of Theorem \ref{thm:critical local limit}]
Consider $\bze$ either 
\begin{itemize}
    \item noncritical, or
    \item critical and such that $\Gamma X_{\bze} \neq \bzero$, where $X_{\bze}$ is the normalized left $1$-eigenvector of $M$.
\end{itemize}
Let $X \in \mathcal{D}_{asy}$ such that $\Gamma X \neq \bzero$, and $(\bk(n))_{n \geq 1}$ as in Theorem \ref{thm:critical local limit}. Necessarily, by assumption, we have $\bzero \notin \mathcal{D}_{asy}^{\Gamma,\bzero}$. By Theorem \ref{thm:accessdir2}, there exists a unique couple $(\btheta,\lambda) \in \cM_{crit} \times (0,+\infty)$ such that $\btheta \sim_\Gamma \bzero$ and $\Gamma X_\btheta=\lambda \Gamma X$. Now, $\bze_\btheta$ satisfies the assumptions of Theorem \ref{thm:critical local limit cas facile} and, by $\Gamma$-equivalence, for all $n \geq 1$: 
\begin{align*}
    \cT^{(1),\btheta}_{\Gamma,\bk(n)} \overset{(d)}{=} \cT^{(1)}_{\Gamma,\bk(n)}.
\end{align*}
The result follows.
\end{proof}

The rest of the section is devoted to the proof of Theorem \ref{thm:critical local limit cas facile}.

Let us start by showing that one can find, for all $n$ large enough, a nonnegative integer preimage of $\bk(n)$ by $\Gamma$ which is asymptotically close to $\{\lambda X, \lambda \in \R\}$.

\begin{Lemma}\label{lem:Blomme}
Let $(\bk(n),n \geq 1)$ be as in Theorem \ref{thm:critical local limit cas facile}. Then, there exists a sequence $(\mathbf{r}(n))_{n \geq 1}$ of elements of $\Z_+^K$ such that:
\begin{itemize}
    \item for all $n \geq 1$, $\Gamma \mathbf{r}(n)^\intercal = \bk(n)^\intercal$ ;
    \item $\big\vert\big\vert\mathbf{r}(n)-||\mathbf{r}(n)||_1X\big\vert\big\vert_1=o(||\mathbf{r}(n)||_1)$.
\end{itemize}.
\end{Lemma}

\begin{proof}
First, observe that $\Gamma(\Z^K)$ is a subgroup of $\Z^\ell$ of rank $\ell$. Let $\alpha_1, \ldots, \alpha_\ell$ be a basis of this subgroup, and $e_1, \ldots, e_\ell$ be preimages of $\alpha_1, \ldots, \alpha_\ell$ by $\Gamma$.
Define $\lambda(n) := \frac{||\bk(n)||_1}{||\Gamma X||_1}$ and $\mathbf{s}(n) := (s_1(n), \ldots, s_K(n)) \in \Z_+^K$ such that, for all $i$, $s_i(n) = \lfloor \lambda(n) X_i \rfloor$. By definition of the sequence $\bk$, 
\begin{align*}
||\Gamma \mathbf{s}(n) - \bk(n)||_1 &\leq ||\Gamma  \mathbf{s}(n) - \lambda(n) \Gamma X||_1 + || \lambda(n) \Gamma X - \bk(n)||_1\\
&=o(\lambda(n)).
\end{align*} 
Now let $c_1(n), \ldots, c_\ell(n) \in \Z$ such that $\Gamma  \mathbf{s}(n) - \bk(n) = \sum_{i=1}^\ell c_i(n) \alpha_i$ (which exists since $\bk(n) \in \Gamma(\Z_+^K)$ by assumption). By equivalence of the norms on $\R^\ell$, there exists $C>0$ such that $\max_i |c_i(n)| \leq C || \Gamma \ \mathbf{s}(n) - \bk(n)||_1=o(\lambda(n))$, and thus
\begin{align*}
\Bigg\vert\Bigg\vert\sum_{i=1}^\ell c_i(n) e_i\Bigg\vert\Bigg\vert_1 = o(\lambda(n)).
\end{align*}

In particular, defining $\mathbf{r}(n) = \mathbf{s}(n) + \sum_{i=1}^\ell c_i(n) e_i$, we have by definition:
\begin{align*}
\Gamma  \mathbf{r}(n) &= \Gamma \mathbf{s}(n) + \sum_{i=1}^\ell c_i(n) \Gamma(e_i)\\
&= \bk(n).
\end{align*}

In addition, since $X$ has positive coordinates and $||\sum_{i=1}^\ell c_i(n) e_i||_1 = o(\lambda(n))$, $\mathbf{r}(n)$ has positive coordinates for $n$ large enough and $||\mathbf{r}(n) - ||\mathbf{r}(n)||_1 X||_1=o(||\mathbf{r}(n)||_1)=o(\lambda(n))$. This ends the proof.
\end{proof}

This allows us to show that the event $\Gamma \ \bN(\cT^{(1)}) = \bk(n)^\intercal$ has large enough probability. 

\begin{Lemma}
\label{lem:probgamma}
   Let $(\bk(n),n \geq 1)$ be as in Theorem \ref{thm:critical local limit cas facile}. We have
    \begin{align*}
        \Prob\left( \Gamma \ \bN(\cT^{(1)}) = \bk(n) \right) = \exp\left(-o(n) \right).
    \end{align*}
\end{Lemma}

\begin{proof}
By Lemma \ref{lem:Blomme}, there exists $(\mathbf{r}(n))_{n \geq 1}$ a equence of elements of $\Z_+^K$ such that $||\mathbf{r}(n) - ||\mathbf{r}(n)||_1 X||_1=o(||\mathbf{r}(n)||_1)$ and $\Gamma \ \mathbf{r}(n) = \bk(n)$.
Using \cite[Theorem 4.7]{ADG18} and the aperiodicity of $\bze$, we get that
\begin{align*}
    \Prob\left( \Gamma \ \bN(\cT^{(1)}) = \bk(n) \right) \geq \Prob\left( \bN(\cT^{(1)})=\mathbf{r}(n)\right) = \exp\left(-o(n) \right).
\end{align*}
\end{proof}

We now prove that the number of vertices of each type in the tree $\cT^{(1)}_{\Gamma, k(n)}$ is concentrated. We say that $f(n)=o_\Prob(g(n))$ if, for any $\varepsilon>0$, $\Prob(|f(n)|>\varepsilon|g(n)|) \rightarrow 0$ as $n \rightarrow \infty$.

\begin{Lemma}
\label{lem:proche de la direction asymptotique}
    Under the assumptions of Theorem \ref{thm:critical local limit cas facile}, for any $i \in [K]$, we have as $n \rightarrow \infty$:
    \begin{align*}
        \left|N_i\left(\cT^{(1)}_{\Gamma, \bk(n)}\right) - ||\bk(n)||_1 X_i \right| = o_\Prob(||\bk(n)||_1).
    \end{align*}
\end{Lemma}

\begin{proof}
Fix $\varepsilon>0$, and consider the event 
\begin{align*}
    E := \left\{ \Gamma \ \bN(\cT^{(1)}) = \bk(n)^{\intercal}, \left|N_i(\cT^{(1)})-\frac{||\bk(n)||_1}{||\Gamma X||_1} X_i\right|>\varepsilon ||\bk(n)||_1  \right\}.
\end{align*}
Write now, for any $\ell \geq 1$, 
\begin{align*}
    E_\ell := E \cap \left\{ N_1(\cT^{(1)})=\ell \right\}.
\end{align*}
For any $\eta>0$, we also define the interval $A_{\eta}(n)$ by:
\begin{equation}
A_{\eta}(n):=\bigg((X_1-\eta)\frac{||\bk(n)||_1}{||\Gamma X||_1},(X_1+\eta)\frac{||\bk(n)||_1}{||\Gamma X||_1}\bigg).
\end{equation}
We can thus write $E=E^{(1)}+E^{(2)}$, where
\begin{align*}
    E^{(1)}=\bigcup_{\substack{\ell \geq 1\\ \ell \not\in A_{\eta}(n)}} E_\ell \text{ and } E^{(2)}=\bigcup_{\substack{\ell \geq 1\\ \ell \in A_{\eta}(n)}} E_\ell,
\end{align*}
for some $\eta>0$ to be fixed later. Observe that
\begin{equation}
\label{eq:equation sur ell}
    \ell \geq (X_1+\eta) \frac{||\bk(n)||_1}{||\Gamma X||_1} \Leftrightarrow \ell-X_1 \frac{||\bk(n)||_1}{||\Gamma X||_1} \geq \frac{\eta}{X_1+\eta} \ell.
\end{equation}
We first have:
\begin{align*}
    \Prob(E^{(1)}) &= \sum_{\substack{\ell \geq 1\\ \ell \not\in A_{\eta}(n)}} \Prob\left( \Gamma \ \bN(\cT^{(1)}) = \bk(n), \left|N_i(\cT^{(1)})-\frac{||\bk(n)||_1}{||\Gamma X||_1} X_i\right|>\varepsilon ||\bk(n)||_1, N_1(\cT^{(1)})=\ell \right)\\
    &\leq \sum_{\substack{\ell \geq 1\\ \ell \not\in A_{\eta}(n)}} \Prob\left( \Gamma \ \bN(\cT^{(1)}) = \bk(n), N_1(\cT^{(1)})=\ell \right)\\
    &\leq \sum_{\substack{\ell \geq 1\\ \ell \not\in A_{\eta}(n)}} \Prob\left( \sum_{i=1}^K \Gamma_{1,i} N_i(\cT^{(1)})=k_1(n), N_1(\cT^{(1)})=\ell \right),
\end{align*}
where, without loss of generality, we assume $\sum_{i=1}^K \Gamma_{1,i} X_i \neq 0$ (since we have assumed $\Gamma X \neq \bzero$).

Consider a sequence of i.i.d. $K$-tuples $\left((\xi_j^{(p)})_{1 \leq j \leq K}\right)_{p \geq 1}$ of random variables distributed as $(N_j^{(1)})_{1 \leq j \leq K}$ (recall the notation from \ref{ssec:asymptotic proportion of types}). Using the cycle lemma (see e.g.~\cite[Lemma 15.3]{Jan12}), we have that, for all $\ell \geq 1$:

\begin{align*}
    &\Prob\left( \sum_{i=1}^K \Gamma_{1,i} N_i(\cT^{(1)})=k_1(n), N_1(\cT^{(1)})=\ell \right)\\
    &= \frac{1}{\ell} \Prob\left(\sum_{p=1}^{\ell} \xi_1^{(p)}=\ell-1, \sum_{p=1}^\ell \sum_{i=1}^K \Gamma_{1,i}\xi_i^{(p)} = k_1(n) \right)\\
    &\leq \Prob\left( \sum_{p=1}^\ell \sum_{i=1}^K \Gamma_{1,i}\xi_i^{(p)} = k_1(n)  \right)\\
    &\leq \Prob\left( \sum_{p=1}^\ell \left( \sum_{i=1}^K (\Gamma_{1,i}\xi_i^{(p)} - \E[\Gamma_{1,i}\xi_i^{(p)}]) \right)=k_1(n)-\ell \frac{\sum_{i=1}^K \Gamma_{1,i} X_i}{X_1} \right) \text{(by Lemma \ref{lem:expectationinblob})}\\
    &\leq \Prob\left( \Big\vert\sum_{p=1}^\ell \left( \sum_{i=1}^K (\Gamma_{1,i}\xi_i^{(p)} - \E[\Gamma_{1,i}\xi_i^{(p)}]) \right)\Big\vert\geq \Big\vert k_1(n)-\ell \frac{\sum_{i=1}^K \Gamma_{1,i} X_i}{X_1}\Big\vert \right).
\end{align*}
The variables $\sum_{i=1}^K (\Gamma_{1,i}\xi_i^{(p)} - \E[\Gamma_{1,i}\xi_i^{(p)}])$ are now centered and it can be checked (see \cite[Prop. 4(i),(iii)]{Mie08}) that they have exponential moments of all order. It implies by a moderate deviation estimate (see e.g. \cite[Example 1.4]{RS15} or \cite[Lemma 3.5]{ABST25+}) that there exists $D>0$ such that, for any $\rho$ small enough, any $x>0$:
\begin{align*}
\Prob\left( \Big\vert\sum_{p=1}^\ell\sum_{i=1}^K (\Gamma_{1,i}\xi_i^{(p)} - \E[\Gamma_{1,i}\xi_i^{(p)}])\Big\vert\geq x \right) \leq \exp \left( D \ell \rho^2 - \rho x \right).
\end{align*}

Observe that $k_1(n) = \frac{||\bk(n)||_1}{||\Gamma X||_1} \sum_{i=1}^K \Gamma_{1,i} X_i +o(||\bk(n)||_1)$ by definition. Taking $\eta>0$ small enough and $\rho=\eta C^{-1}$ for $C>0$ such that $D X_1(X_1+\eta)-C |\sum_{i=1}^K \Gamma_{1,i} X_i|<0$ (which exists since we have assumed $\sum_{i=1}^K \Gamma_{1,i} X_i \neq 0$), we get that:

\begin{align*}
\Prob\left( E^{(1)}\right) &\leq\sum_{\substack{\ell \geq 1\\ \ell \not\in A_{\eta}(n)}} \Prob\left( \Big\vert\sum_{p=1}^\ell \left( \xi_i^{(p)} - \E[\xi_i^{(p)}] \right)\Big\vert\geq \Big\vert k_1(n)-\ell\frac{\sum_{i=1}^K \Gamma_{1,i} X_i}{X_1}\Big\vert \right) \\  
&\leq \sum_{\substack{\ell \geq 1\\ \ell \not\in A_{\eta}(n)}} \exp\left( D\ell\frac{\eta^2}{C^2} - \frac{\eta}{C} \frac{|\sum_{i=1}^K\Gamma_{1,i} X_i|}{X_1} \left| \ell-X_1 \frac{||\bk(n)||_1}{||\Gamma X||_1} +o(||\bk(n)||_1) \right| \right)\\
&\leq \sum_{\ell \geq (X_1+\eta)\frac{||\bk(n)||_1}{||\Gamma X||_1}} \exp\left( D\ell\frac{\eta^2}{C^2} - \frac{|\sum_{i=1}^K\Gamma_{1,i} X_i|}{X_1} \frac{\eta}{C} \left( \ell-X_1 \frac{||\bk(n)||_1}{||\Gamma X||_1} +o(||\bk(n)||_1) \right) \right)\\
&\qquad \qquad \qquad + \left( X_1-\eta \right) \frac{||\bk(n)||_1}{||\Gamma X||_1} \exp\left( D X_1 \frac{||\bk(n)||_1}{||\Gamma X||_1}\frac{\eta^2}{C^2} - \frac{\eta^2}{C} \frac{\sum_{i=1}^K|\Gamma_{1,i} X_i|}{X_1}  \frac{||\bk(n)||_1}{||\Gamma X||_1}   \right) \\
&\leq \sum_{\ell \geq (X_1+\eta)\frac{||\bk(n)||_1}{||\Gamma X||_1}} \exp\left( D\ell\frac{\eta^2}{C^2} - \frac{|\sum_{i=1}^K\Gamma_{1,i} X_i|}{X_1} \frac{\eta}{C} \frac{\eta}{X_1+\eta}\frac{||\bk(n)||_1}{||\Gamma X||_1}+o(||\bk(n)||_1) \right)\\
&\qquad \qquad \qquad + \left( X_1-\eta \right) \frac{||\bk(n)||_1}{||\Gamma X||_1} \exp\left( \eta^2 \frac{D X_1^2-C |\sum_{i=1}^K \Gamma_{1,i} X_i|}{2X_1 C^2} \frac{||\bk(n)||_1}{||\Gamma X||_1} +o(||\bk(n)||_1)\right) \\
&\leq \exp\left(-B||\bk(n)||_1 \right)
\end{align*}
for some constant $B>0$, using \eqref{eq:equation sur ell} and the definition of $C$. This holds for $\eta>0$ fixed, for all $n$ large enough.

We can now choose $\eta>0$ small enough so that $\eta \frac{X_i}{X_1||\Gamma X||_1}<\frac{\varepsilon}{2}$. Using the same technique as for $E^{(1)}$, we can write that:
\begin{align*}
    \Prob\left( E^{(2)} \right) &\leq \sum_{\substack{\ell \geq 1 \\ \ell\in A_{\eta}(n)}} \Prob\left( \Gamma \ \bN(\cT^{(1)}) = \bk(n), \left|N_i(\cT^{(1)})-\frac{||\bk(n)||_1}{||\Gamma X||_1} X_i\right|>\varepsilon ||\bk(n)||_1, N_1(\cT^{(1)})=\ell \right)\\
    &\leq \sum_{\substack{\ell \geq 1\\ \ell\in A_{\eta}(n)}} \Prob\left(  \left|N_i(\cT^{(1)})-\frac{||\bk(n)||_1}{||\Gamma X||_1} X_i\right|>\varepsilon ||\bk(n)||_1, N_1(\cT^{(1)})=\ell \right)\\
    &\leq \sum_{\substack{\ell \geq 1 \\ \ell\in A_{\eta}(n)}} \sum_{\left|m-\frac{||\bk(n)||_1}{||\Gamma X||_1} X_i\right| \geq \varepsilon ||\bk(n)||_1} \Prob \left( \left| \sum_{p=1}^\ell (\xi_i^{(p)} - \E[\xi_i^{(p)}]) \right| \geq \left| m-\ell \frac{X_i}{X_1} \right| \right)\\
    &\leq \sum_{\substack{\ell \geq 1 \\ \ell\in A_{\eta}(n)}} \sum_{\left|m-\ell\frac{X_i}{X_1}\right| \geq \frac{\varepsilon}{2} ||\bk(n)||_1} \Prob \left( \left| \sum_{p=1}^\ell (\xi_i^{(p)} - \E[\xi_i^{(p)}]) \right| \geq \left| m-\ell \frac{X_i}{X_1} \right| \right)\\
    &\leq \exp\left( - B' ||\bk(n)||_1 \right)
\end{align*}
for some $B'>0$, by \eqref{eq:grandes dev}. Indeed, the choice of $\eta$ ensures that we have, for $m$ such that $\left|m-\frac{||\bk(n)||_1}{||\Gamma X||_1} X_i\right| \geq \varepsilon ||\bk(n)||_1$: 

\begin{align*}
\left|m-\ell \frac{X_i}{X_1}\right| &\geq \left|m-\frac{||\bk(n)||_1}{||\Gamma X||_1} X_i\right| - \left|\frac{||\bk(n)||_1}{||\Gamma X||_1} X_i-\ell \frac{X_i}{X_1}\right|\\
&\geq \varepsilon ||\bk(n)||_1 - \eta  \frac{X_i}{X_1} \frac{||\bk(n)||_1}{||\Gamma X||_1} \text{ since }\ell \in A_\eta(n)\\
&\geq \frac{\varepsilon}{2} ||\bk(n)||_1 \text{by our choice of } \eta.
\end{align*}
Thus, we have
\begin{align*}
    \Prob\left( \left| N_i\left( \cT^{(1)}_{\Gamma,\bk(n)}\right)- ||\bk(n)||_1 X_i\right| > \varepsilon ||\bk(n)||_1\right) \leq \exp(-B'' ||\bk(n)||_1),
\end{align*}
for some $B''>0$. 

We conclude using Lemma \ref{lem:probgamma}.
\end{proof}

We can now prove Theorem \ref{thm:critical local limit cas facile}.

\begin{proof}[Proof of Theorem \ref{thm:critical local limit}]
For all $n \geq 1$, let $I_{n,\varepsilon} := \{\bn := (n_1, \ldots, n_k)^{\intercal} \ | \ \Gamma \ \bn=\bk(n) \text{ and}, \ \forall i \in [K], |n_i-X_i||\bn||_1| \leq \varepsilon ||\bn||_1\}$. By Lemma \ref{lem:proche de la direction asymptotique}, there exists $\varepsilon_n \rightarrow 0$ such that $\Prob\left( \bN(\cT_{\Gamma,\bk(n)}^{(1)}) \in I_{n,\varepsilon_n} \right) \rightarrow 1$. We conclude by \cite[Theorem 3.1]{ADG18}.
\end{proof}

\section{Appendix: a counter-example when $\bmu$ is not entire}
\label{sec:appendix}

We construct here an example of a supercritical probability measure that does not admit a critical equivalent. No such example exists in the monotype case, showing that the multitype case is more intricate. 

The idea is as follows: start from an entire supercritical distribution, and add a small well-chosen non-entire perturbation. It turns out that one can do it without changing too much the functions $\phi^{(i)}$ within the domain of convergence. By doing this, we can make some critical tiltings "leave" the domain of convergence. We consider the following function, defined on $(-\infty,0)$:

\begin{equation*}
    g(\theta):=\sum_{n\geq 3} \frac{2}{n(n-1)(n-2)}e^{n\theta}
\end{equation*}

It has the following properties.

\begin{Lemma}
\label{lem:fonction g}
We have, for any $\theta < 0$:
\[
g(\theta)=-\left(e^{\theta}-1\right)^2\log\left(1-e^{\theta}\right)+e^{\theta}\left(\frac{3}{2}e^{\theta}-1\right).
\]
In addition, for all $\theta\leq 0$, we have $0\leq g(\theta),g^\prime(\theta) \leq 2$ and
\begin{align*}
    g(\theta),g'(\theta) \underset{\theta \rightarrow-\infty}{\rightarrow} 0.
\end{align*}
\end{Lemma}
\begin{proof}
It suffices to observe that, for all $n \geq 3$, $$\frac{2}{n(n-1)(n-2)}=\frac{1}{n}-\frac{2}{n-1}+\frac{1}{n-2},$$
and check that
\[
g'(\theta)=-2e^{\theta}(e^{\theta}-1)\log(1-e^{\theta})+2e^{2\theta}.
\]
\end{proof}

We now consider, for any $\varepsilon,A>0$, the $2$-type tree whose two offspring distributions have generating functions:
\[
\begin{cases}
\phi^{(1)}(e^{\btheta}) =&\frac{1}{2+2e^{-2A}+\varepsilon g(-1)} \left(1 + 2e^{2(\theta_1-A)} + e^{2\theta_2} +\varepsilon g(\theta_1-1)\right)\\
\phi^{(2)}(e^{\btheta}) =& \frac{1}{2+e^{-2A}}\left(1 + e^{2(\theta_1-A)} + e^{2\theta_2}\right)\\
\end{cases}
\]

Denote by $\bmu^{A,\varepsilon}$ the associated projection of offspring distributions, and by $\chi^{A,\varepsilon}$ the function associated to $\bmu^{A,\varepsilon}$ by Definition \ref{def:functionc}.

\begin{Proposition}
\label{prop:counterexample}
    For a good choice of $\btheta, A, \varepsilon$ and for $\Gamma := (1\ 1)$, $\bmu_{\btheta}^{A,\varepsilon}$ is supercritical, but has no critical $\Gamma$-equivalent tilting.
\end{Proposition}

In order to prove this, we start with a first lemma.

\begin{Lemma}
\label{lem:counterexample}
For any $c_1,C_2>0$, there exist $A_0,\varepsilon_0>0$ such that, for any $A>A_0$, for any $\varepsilon \in (0,\varepsilon_0)$, for any $\btheta$ for which $\bmu^{A,\varepsilon}_{\btheta}$ is critical, we have $\theta_1-A<-C_2$ and $|\theta_2|<c_1$.
\end{Lemma}

\begin{proof}
Fix $A, \varepsilon>0$. For any $\btheta \in \R^2$ critical, $1$ is an eigenvalue of $M_\btheta$ and we have:
\[
\left(\frac{4e^{2(\theta_1-A)}+\varepsilon g^\prime(\theta_1-1)}{\gamma^{(1)}(\btheta)}-1\right)\left(\frac{2e^{2\theta_2}}{\gamma^{(2)}(\btheta)}-1\right)-\frac{2e^{2\theta_2}}{\gamma^{(1)}(\btheta)}\frac{2e^{2(\theta_1-A)}}{\gamma^{(2)}(\btheta)}=0,
\]
where
\[
\begin{aligned}
\gamma^{(1)}(\btheta) :=&1 + 2e^{2(\theta_1-A)} + e^{2\theta_2} +\varepsilon g(\theta_1-1),\\
\gamma^{(2)}(\btheta) :=& 1 + e^{2(\theta_1-A)} + e^{2\theta_2}.
\end{aligned}
`\]
This can be rewritten as:
\[
\left(4e^{2(\theta_1-A)}+\varepsilon g^\prime(\theta_1-1)-\gamma^{(1)}(\btheta)\right)\left(2e^{2\theta_2}-\gamma^{(2)}(\btheta)\right)=4e^{2(\theta_1-A)+2\theta_2}.
\]
Using the notation $X=e^{2(\theta_1-A)},Y=e^{2\theta_2}$ and $\delta=g(\theta_1-1)-g^\prime(\theta_1-1)$ we get:
\[
\left(2X-1-\varepsilon\delta-Y\right)\left(Y-1-X\right)=4XY.
\]
This can be rewritten as:
\begin{equation}\label{eq:critcounter1}
Y^2+(X+\varepsilon \delta)Y-(1+X)(1-2X+\varepsilon \delta)=0.    
\end{equation}
This is equivalent to:
\begin{equation}\label{eq:critcounter2}
\bigg(Y+\frac{X+\varepsilon \delta}{2}\bigg)^2=\frac{(X+\varepsilon \delta)^2}{4}+(1+X)(1-2X+\varepsilon \delta)
\end{equation}

Observe that the radius of convergence of $g \circ \log$ is $1$. Hence, if $\btheta := (\theta_1, \theta_2)$ is critical, we have $\theta_1<1$ and in particular, if $A$ is large enough so that $A\geq C_2+1$, we have $\theta_1-A\leq -C_2$. Using the fact that $g,g'$ are bounded on $\R_-$ (Lemma \ref{lem:fonction g}), we get that $\delta$ is bounded. We have that $X\leq e^{-2(A-1)}$. For $X,\varepsilon>0$ small enough, \eqref{eq:critcounter2} has one positive solution $Y(X,\varepsilon)$ and, as $X,\varepsilon \rightarrow 0$, $Y(X,\varepsilon) \rightarrow 1$. This means that by taking $A$ large enough and $\varepsilon$ small enough we can ensure that any positive $Y$ that satisfies \eqref{eq:critcounter2} also satisfies $|\log(Y)|<c_1$. The result follows.
\end{proof}

Let us show how it implies Proposition \ref{prop:counterexample}.

\begin{proof}[Proof of Proposition \ref{prop:counterexample}]
Observe that, for any choice of $A,\varepsilon$, for any $s$ large enough, $\phi^{(1)},\phi^{(2)}$ are well-defined at $(-s,s)$. The constants $A$ and $\varepsilon$ being fixed, as $s$ goes to $+\infty$, we can compute the following asymptotics:
\[
\chi^{A,\varepsilon}(-s,s)=(3s,s) + O_s(1).
\]

Using this along with Lemma \ref{lem:counterexample}, for any $A>0$ large enough, any $\varepsilon>0$ small enough, there exists a constant $C:=C(A,\varepsilon)>0$ such that we have:
\begin{itemize}
\item[(a)] $||\chi^{A,\varepsilon}(-s,s)-(3s,s)||_1\leq C$,
\item[(b)] if $\bmu_\btheta^{A,\varepsilon}$ is critical then $\theta_1-A \leq 0$ and $||\chi^{A,\varepsilon}(\btheta)-(-\theta_1,0)||_1\leq C$.
\end{itemize}

Now fix $A$ large enough, $\varepsilon>0$ small enough and $C>0$ such that (a) and (b) hold, and take $\Gamma=(6\ 1)$. We will prove that, for $s>0$ large enough, $\bmu_{(-s,s)}^{A,\varepsilon}$ has no critical $\Gamma$-equivalent tilting. To this end, fix $s>0$ and assume that there exists $\btheta$ critical and $\lambda\in\R$ such that $\chi^{A,\varepsilon}(-s,s)=\chi^{A,\varepsilon}(\btheta)+\lambda \Gamma$.
It follows from (a) and (b) that $|\lambda-s/2-\theta_1/6|\leq C/3$ (checking the first coordinate), and that $|\lambda-s|\leq 2C$ (checking the second coordinate). From this we get that $|s/2-\theta_1/6|\leq C/3$. Since we necessarily have $\theta_1\leq 1$ (by definition of $g$), we obtain that, for $s$ large enough, there is no solution. 
To conclude, observe that, for $s$ large enough, the measure is supercritical. Indeed, by definition, a node of type $2$ has two children of type $2$ with probability arbitrarily close to $1$ as $s \rightarrow +\infty$.
\end{proof}

\section{Appendix: number of pre-images for $\chi$}
\label{sec:appenonze}
We know by Proposition \ref{prop:bijection} that $\chi$ induces a bijection between $\cM_{crit}$ and the boundary of $\mathcal{C}_{image}$. 
We also know by Lemma \ref{lem:fromsupercrittosubcrit} that an element in the interior of $\mathcal{C}_{image}$ may have multiple pre-images by $\chi$. It turns out that there is actually no upper bound to the number of pre-images of an element by $\chi$. \\
To show this we consider the following example for some $N \geq 3$:
\begin{itemize}
    \item There are $K:=2N$ types,
    \item a vertex of type $i \in [2N]$ has no child with probability $1/4$, 2 children of type $i$ with probability $1/2$ or 1 child of each type (so, $2N$ children in total) with probability $1/4$.
\end{itemize}
In this setting we have for all $\btheta \in \R^{2N}$ and for all $i \in [2N]$:
\begin{equation}
    \phi^{(i)}(e^\btheta)= \frac{1+e^{\sum_k \theta_k}+2e^{2\theta_i}}{4}.
\end{equation}
The reason we look at this case is that the Jacobian of $\chi$ at $\btheta=\bzero$ is equal to the matrix with $1/4$ on every entry. This means that it should be far from locally invertible and we can hope that some values have many pre-images by $\chi$.\\
To simplify notations, write $\gamma:=e^{\sum_k \btheta_k}$. We want to find values of $\btheta \in \R^K$ for which $\chi(\btheta)=\bzero$.

To this end, let $\delta$ be a solution in $(0,1)$ of
\begin{align*}
    \delta = \left( \frac{1+\delta}{2} \right)^N.
\end{align*}

Let $I$ be a set of $N$ elements of $[2N]$, and let $\btheta := (\theta_1, \ldots, \theta_{2N})$ be such that
$$
e^{\theta_i} = \left\{
    \begin{array}{ll}
       1+\sqrt{\frac{1-\delta}{2}} & \mbox{if } i \in I \\
        1-\sqrt{\frac{1-\delta}{2}} & \mbox{otherwise.}
    \end{array}
\right.
$$

In this case, we get

\begin{align*}
    e^{\sum_k \theta_k}
    &= \left( 1+\sqrt{\frac{1-\delta}{2}} \right)^N \left( 1-\sqrt{\frac{1-\delta}{2}} \right)^N\\
    &= \left( \frac{1+\delta}{2} \right)^N\\
    &= \delta
\end{align*}

In addition, for all $i \in [2N]$, we have
\begin{align*}
  \frac{1}{2}  e^{2\theta_i} - e^{\theta_i} + \frac{1+\delta}{4}=0.
\end{align*}

This can be rewritten as:

\begin{align*}
  \forall i \in [2N], \frac{1+e^{\sum_k \theta_k}+2e^{2\theta_i}}{4}=e^{\theta_i}.
\end{align*}

We get from that:

\begin{align*}
\forall i \in [2N], \ e^{\chi_i(\btheta)} 
&= \frac{1+e^{\sum_k \theta_k}+2e^{2\theta_i}}{4}e^{-\theta_i}\\
&= \frac{1+\delta+2e^{2\theta_i}}{4}e^{-\theta_i}\\
&= 1,
\end{align*}
so that $\chi(\btheta)=\bzero$. Since there are $\binom{2N}{N}$ choices for the set $I$, we get that $\bzero$ has at least $\binom{2N}{N}$ pre-images by $\chi$.

\bibliographystyle{abbrv}
\bibliography{main}
\end{document}